\documentclass[12pt,a4paper]{amsart}
\usepackage{preamble}

\title[fast differentiation of chaos]{
Fast differentiation of hyperbolic chaos
}

\begin{document}

\begin{abstract}
We derive and prove the `fast response' formula for the linear response, the parameter derivatives of long-time-averaged statistics, of hyperbolic deterministic chaotic systems.
The expression is pointwisely defined so we can compute the linear response in high-dimensions via Monte-Carlo-type algorithms.
It has two parts, where the shadowing contribution is computed by the nonintrusive shadowing algorithm. 
The unstable contribution is expressed by renormalized second-order tangent equations; importantly, it does not contain any distributional derivatives.
The algorithm's cost is solving $u$, the unstable dimension, many first-order and second-order tangent equations along a long orbit; the main error is the sampling error of the orbit. 
We numerically demonstrate the algorithm on a 21-dimensional example, which is difficult for previous methods.


\end{abstract}

\maketitle

\section{Introduction}

\subsection{Main results of the paper}\footnote{This paper was first available on arXiv on September 1, 2020.}
\hfill\vspace{0.1in}
\label{s:main results}

It is an open problem to devise a precise and efficient algorithm for the linear response for high-dimensional hyperbolic chaotic systems.
Here, the linear response is the derivative of the SRB/stationary measure with respect to some parameters in the dynamics.
The main bottleneck is to find a pointwise expression (called an `estimator') for the unstable divergence, $\div^u X^u$, where $X^u$ is the unstable part of the perturbation on the dynamics.
This paper gives a solution to this problem through the following steps.

\begin{enumerate}
  \item 
  Get rid of derivatives of $X^u$ in \cref{l:hen} of \cref{s:transform2pi} by rewriting the unstable divergence as the derivative of a volume ratio $\varpi$ along the stable foliation.
  \item
  Derive an pointwisely defined expansion formula for the derivative of $\varpi$ in \cref{l:hao} of \cref{s:transform3}.
  \item 
  Define a big term, $p$, as the sum of small terms in the expansion.
  Then express the unstable contribution by $p$ in \cref{l:UC by p} of \cref{s:cai}.
  \item 
  Find the dynamics of $p$ in \cref{l:induction} in \cref{s:fastChar}, and show that $p$ is the unique limit of this dynamics along any orbit.
\end{enumerate}

Hence, we obtain the fast response formula in \cref{t:lra}, which is the integration of a pointwise-defined estimator.
The estimator is expressed by the solution of only $u$ many first- and second-order tangent equations, which can be defined recursively on an orbit.
More specifically, for any $r_0$, 
  \[ \begin{split}
    U.C.^W 
    = \lim_{N \rightarrow \infty} \frac 1 {N} \sum_{n=0}^{N-1}
      \ip{\tilde \beta r_n ,
      \tilde e_{n+1} },
    \quad\textnormal{where}\quad
    r_{n+1}= P^\perp \tilde \beta r_n .
  \end{split} \]
Here $\ip{\cdot, \cdot}$ is the Riemannian metric defined in equation~\eqref{e:tensorproduct}.
Roughly speaking, $\tilde e$ is the unit $u$-dimensional cube spanned by unstable vectors,
$r_n\in\cD^u$ are derivatives of cubes (see appendix~\ref{a:derivative}) that converges to $p$;
$\tilde \beta$ is the renormalized second-order tangent equation 
governing the propagation of derivatives of vectors (definition~\ref{d:betaU} in section~\ref{s:betaU});
$P^\perp$ is the orthogonal projection operator on $\cD^u$ (appendix~\ref{a:projection}).
A side benefit is that we get rid of oblique projections (the projections onto the stable or unstable subspaces); this will bring some more efficiency and robustness to the algorithm.

The second main result is the fast response algorithm, which is precise, efficient in high-dimensions, and easy to implement.
The algorithm works well in high-dimensions since it is Monte-Carlo: it first computes an orbit, then inductively solves several quantities on the orbit, then computes an estimator from these quantities, and takes average of the estimator along the orbit.
The inductive nature makes the algorithm quite efficient; its cost is only $O(Mu)$ per step, where $M$ is the system dimension and $u$ is the unstable dimension.
We further write everything in matrix notation in section~\ref{s:procedure list}, which is more suitable for coding.
The algorithm requires little additional coding to first-order and second-order tangent solvers, where second-order tangent solvers are just first-order solvers with a second-order source term.

Section~\ref{s:example} shows a numerical application on a modified solenoid map,
which is a 21-dimensional system with a 20-dimensional unstable subspace, whose direction is unknown beforehand.
For the same precision, the fast response is about $10^{6}$ times faster than the path-perturbation method.
Also, the 21 dimension would require a tremendous number of elements for conventional operator-based, or divergence, algorithms.
The fast response algorithm is even faster than the regression methods.

The last result of this paper is several basic geometry and algebra tools
for the second-order tangent equation,
which governs the propagation of derivatives of unstable $u$-vectors.
These tools are useful when considering $u$-dimensional invariant submanifolds.
In particular, appendix~\ref{a:f*} defines the derivative of the pushforward operator $\nabla f_*$.
Appendix~\ref{a:derivative} defines the linear space, $\cD^u$, of derivatives of unstable $u$-vectors.
Appendix~\ref{a:projection} extends the projection operators on single-vectors to $\cD^u$.

This paper is organized as follows.
First, we review the related literature.
Section~\ref{s:prep} recalls in detail the tools that we shall use, such as hyperbolic systems,
the linear response formula, and the nonintrusive shadowing algorithm.
Section~\ref{s:3step} gets rid of the distributional derivatives in the unstable divergence, and derives a pointwisely defined expansion formula.
Section~\ref{s:compute unstable} derives a convergent inductive expression from the expansion.
Section~\ref{s:algorithm} gives more details of the fast response algorithm and lists the pseudocode.
\Cref{s:discuss} gives some discussions on coding, cost estimation, and a potential program to extend to non-hyperbolic systems.
Section~\ref{s:example} shows a numerical application on a 21-dimensional example.

\subsection{Literature review}
\hfill\vspace{0.1in}
\label{s:review1}

Chaos appears in many disciplines, such as fluid mechanics, geophysics, and machine learning. 
For a deterministic map on a manifold, the physical measure, or SRB measure, gives the long-time-averaged statistics of chaotic systems.
The linear response is the derivative of an observable function, integrated according to a marginal or stationary/SRB measure, with respect to some parameters in the dynamics, or the governing equation.
It is fundamental to many numerical tools widely used in those disciplines, such as gradient-based optimization, error analysis, and uncertainty quantification.
However, computing the linear response is challenging.
There are mainly two methods for the linear response of deterministic systems: the path perturbation method and the divergence method.

The path-perturbation method is formally the average of the orbit-wise perturbations over many orbits (see \cref{s:linResponse}).
Theoretically, this gives an accurate linear response for hyperbolic systems and is the final formula presented in the proof \cite{Ruelle_diff_maps,Ruelle_diff_maps_erratum,Ruelle_diff_flow,Jiang2012,Dolgopyat2004}.
Numerically, the path-perturbation algorithm has been developed in \cite{Lea2000,eyink2004ruelle,lucarini_linear_response_climate,lucarini_linear_response_climate2}.
However, because the orbit-wise perturbations grow exponentially fast, it is typically unaffordable for convergence to actually happen.
The path-perturbation method also includes the backpropagation method, which is the basic algorithm for machine learning.

The divergence method is also known as the (measure) transfer operator method, since the perturbation of the transfer operator is some kind of divergence (see \cite{TrsfOprt} for a derivation).
Theoretically, this is also correct for hyperbolic systems, and it is the final formula presented in the proofs in \cite{Gouezel2006,Bahsoun2018,Crimmins2020,optimalresponse2022,SantosGutierrez2020}.
Numerically, it has been developed in \cite{Galatolo2014,Wormell2019a,Zhang2020}.
However, it is expensive for high-dimensional systems with contracting directions for the reasons below.

To work in high-dimension, it seems that we must use a Monte-Carlo type linear response formula/algorithm.
This means that the formula is the average of an expression (called an `estimator'), which can be evaluated on each point of an orbit, and then we can take average of the estimator over a long orbit or several orbits.
When the system has contractions, the divergence formula becomes distributions, so it can not be evaluated pointwisely.
Unlike the unstable divergence which we shall discuss later, the full divergence here is a distribution by itself, so there is no hope to compute it pointwisely.
The work-around is to avoid the Monte-Carlo approach, and use a finite set of basis functions to approximate singular SRB measures, but the cost increases exponentially fast with the dimension (see \cite{TrsfOprt} for a cost estimation).

Mixing two basic methods can overcome some of their major shortcomings.
This idea has been seen in various theoretical works mentioned above, although they did not end up giving new formulas.
In numerics, the blended response algorithm by Majda and Abramov first attempted to combine the path-perturbation and the divergence method.
Its path-perturbation part has cost $O(M^2)$ per time step, where $M$ is the dimension of the system; this cost is equal to or lower than some related and later algorithms such as least-squares shadowing \cite{Chater_convergence_LSS,Blonigan_MSS,Shawki2019,Lasagna2019}, but it is not optimal.
Its divergence part was approximated by a degenerate finite-element type algorithm, where the measure is approximated by a single Gaussian function, whose accuracy can not be improved without exponentially increasing the cost \cite{Abramov2008}.

This paper gives a complete and efficient solution to mixing the path-perturbation and the divergence method, which we call the path-divergence method, or the fast response method.
More specifically, we decompose the linear response into the shadowing contribution and the unstable contribution, which are then expressed by new path-perturbation and divergence formulas, respectively.

The notion of shadowing can be regarded as a particular path-perturbation which remains bounded in hyperbolic systems \cite{Ano_shadow,Bowen_shadowing,Pilyugin_shadow_linear_formula}.
We use our `nonintrusive-shadowing' algorithm for this part of the linear response, which constrains the computation to only the unstable subspace.
Hence, its cost is reduced to only $O(Mu)$ per step, where $u$ is the unstable dimension \cite{Ni_NILSS_JCP,Ni_fdNILSS}.
When $u\ll M$, the unstable contribution can be small \cite{Ruesha}, hence, we can sometimes approximate the linear response by only the shadowing contribution.
Nonintrusive shadowing could compute a reasonable linear response in practical problems such as computational fluid systems with $4\times10^6$ degrees of freedom \cite{Ni_CLV_cylinder}.
Moreover, this paper will show that nonintrusive shadowing is also important in computing the unstable contribution.

The unstable contribution can be expressed by the derivative of the transfer operator on unstable manifolds, which is further expressed by a divergence in the unstable manifold, $\div^uX^u$ (see \cref{s:transform1}, and \cite{TrsfOprt}).
But this expression can be deceptive since, although $\div^uX^u$ as a whole is a function, the directional derivatives of $X^u$ are true distributions (not functions) so they are not pointwisely defined.
In numerics, this means that we get infinite values when trying to compute the directional derivatives, so we can not run Monte-Carlo-type algorithms, and we face the same difficulty as the conventional divergence method.
One goal of this paper is to find another expression of the unstable divergence which does not secretly involve distributions.
In other words, we find a pointwisely defined expression of the derivative of transfer operators on unstable manifolds, so that we can run Monte-Carlo algorithms to compute it in high-dimensions.

\subsection{Review of literature after this paper}
\hfill\vspace{0.1in}
\label{s:review update}

During the handling process of this paper, we gave several extensions of the fast response method and several new linear response methods.
First, we gave an adjoint version of the method, whose cost is almost independent of the number of parameters; its theory is also a bit easier to understand than the current paper \cite{Ni_asl,TrsfOprt,Ni_nilsas,far}.
We also gave the fast adjoint response method for continuous-time hyperbolic systems \cite{vdivF}.
We also used the fast adjoint response method to solve the optimal response problem in the hyperbolic setting \cite{GN25}.

In particular, our adjoint version of the fast response formula in \cite{Ni_asl,TrsfOprt} has made much improvement over the current paper in terms of understandability.
This is because the current paper only has an expansion formula, but not a fast formula for unstable divergence $\div^u X^u$.
In this paper, we first expand $\div^u X^u$ into many small terms, then use the invariance of SRB measure to change the time steps of each small term, yielding an `asynchronous' version of $\div^u X^u$, which is highly-related to but is not exactly $\div^u X^u$.
On the other hand, our fast adjoint response formula was able to give an explicit fast formula of $\div^u X^u$: with this intermediate step as anchor, the fast adjoint response formula should be easier to understand, and its relation to the transfer operator on unstable manifolds is also more clear than the current paper.

The fast response method does not work when hyperbolicity is poor.
In fact, we can show that many simple systems do not have linear responses \cite{Baladi2007,wormell22}.
A possible solution is to modify the dynamics a bit so that we can compute the linear response of the modified system, which still offers some guidance for the optimization of the original system.
The plausible modification is to add noise.

Random dynamical systems has another basic linear response formula, the kernel-differentiation formula.
Its main feature is that the derivative hits the probability kernel at each step; since the dynamics is not differentiated, it is not affected by the lack of hyperbolicity.
It is also known as the Cameron-Martin-Girsanov theorem \cite{CM44, MalliavinBook} or the likelihood ratio method \cite{Rubinstein1989,Reiman1989,Glynn1990}.
Proofs of the conventional kernel-differentiation formula for the case of stationary measures are given, for example, in \cite{HaMa10}.
We gave two new kernel-differentiation formulas, an ergodic version and a foliated version, in \cite{Ni_kd}.
However, the kernel-differentiation method in general cannot handle multiplicative noise or perturbation of the diffusion coefficients; it is also expensive when the noise is small.

Mixing the kernel-differentiation method with the path-perturbation or divergence method can also overcome some major shortcomings of basic methods.
This was achieved by the author's path-kernel method in \cite{dud,apk} and the divergence-kernel method in \cite{divKer,DKlinR}.
Both work for non-hyperbolic systems and multiplicative noise controlled by some parameters.
In particular, the adjoint path-kernel method can solve a difficult version of the variational data assimilation problem, where we optimize an unstable diffusion process to match a low-dimensional observation path \cite{apk}.
On the other hand, the divergence-kernel method computes the derivative of the marginal or stationary density; hence, we can directly optimize the diffusion process so that the marginal distribution matches the data -- this gives a new framework of generative models \cite{DKlinR}.

In summary, we have obtained all combinations of two (out of three) basic linear response methods.
They have significantly improved our capacity of dealing with chaos in all related fields such as fluids and machine learning, while still have some limitations.
Hence, it is likely that we should seek to mix all three basic methods.
This 3-in-1 program was proposed in \cite{Ni_kd} and is also briefly mentioned in \cref{s:beyond hyper}.

\section{Preparations}
\label{s:prep}

\subsection{Hyperbolicity and notations}
\hfill\vspace{0.1in}
\label{s:hyperNotation}

Let $f$ be a $C^3$ map on a $C^\infty$ Riemannian manifold $\cM$, whose dimension is $M$.
Assume that $K$ is a compact hyperbolic invariant set, that is, $T_K\cM$ (the tangent bundle restricted to K) 
has a continuous $f_*$-invariant splitting $T_K\cM = V^s \bigoplus V^u$,
such that there are constants $C>0$, $0<\lambda < 1$, and
\[
  \max_{x\in K}\|f_* ^{-n}|V^u(x)\| ,
  \|f_* ^{n}|V^s(x)\| \le C\lambda ^{n} \quad \textnormal{for all}\quad n\ge 0.
\]
Here $f_*$ is the pushforward operator, which is the Jacobian matrix of $f$ when $\cM=\R^M$ (see definition~\ref{d:f*} in appendix~\ref{a:f*}).
We call $V^u$ and $V^s$ the stable and unstable  subspaces.
Let $u, s$ denote the dimension of the unstable and stable manifolds, so $u+s=M$.
A stable manifold, $\mathcal V^s(x)$, is as smooth as $f$, tangent to $V^s(x)$ at $x$,
and there are $C>0$ and $\lambda< 1$ such that if $y,z\in\cV^s(x)$,
\[
  d(f^n y, f^n z) \le C \lambda^{n} d (y,z) \quad \textnormal{for}\quad n\ge 0,
\]
where $d$ is the distance function.
Unstable manifolds are defined similarly.

We assume that $K$ is an Axiom A attractor, which is the closure of periodic orbit,
and there is an open neighborhood $U$, called the basin of the attractor, such that $\cap_{n\ge0} f^nU=K$.
For all $x\in K$, the local unstable manifolds $\cV^u(x)$ are in $K$,
whereas the local stable manifolds $\cV^s(x)$ fill a neighborhood of $K$.
For more details on hyperbolicity, see \cite{Ruelle1989,Shub1987}.
Due to the spectral decomposition theorem, taking a basic set and raise $f$ to some power,
we may further assume that $f$ is mixing on $K$ \cite{Bowen,Smale1967}.

Under our assumptions, the SRB measure $\rho$ of $f$ on $K$ is the unique $f$-invariant measure
with either of the following characterizations \cite{young2002srb}:
\begin{itemize}
  \item $\rho$ has absolute-continuous conditional measures on unstable manifolds;
  \item $\rho$ is the physical measure, that is, there is a set $V\subset \cM$ having a full Lebesgue measure such that for
  every continuous observable $\phi:\cM \rightarrow \R$, almost all $x\in V$
  \[ \begin{split}
    \frac 1n \sum _{i=0} ^{n-1} \phi (f^ix) \rightarrow \rho(\phi).
  \end{split} \]
\end{itemize}
Note that the attractor $K$ is typically lower-dimensional with a zero Lebesgue measure: this was proved for many important systems such as the Navier-Stokes and SQG equation \cite{Constantin1985a,Constantin2016}.
The fractal nature of the attractor causes major difficulties.

We explain some notation conventions used in this paper.
Let $x\in \cM$ be the point of interest and $x_k:=f^kx$.
The subscripts $n, k, m$ only label the steps and the subscripts for step zero are omitted.
For a tensor field $X$ on $\cM$, let $X_k$ be the pullback of $X$ by $f^k$,
\[ \begin{split}
  X_k(x):=X(x_k).
\end{split} \]
The subscript $k$ always specifies that the value of the tensor field is taken at $x_k$.
However, when $X_k$ is differentiated, we should further specify its domain or when the pullback occurs: we leave that to be determined by the differentiating vector.
For example, let $\nabla$ be the Riemannian connection, $Y$ a vector field, then
\begin{equation} \label{e:nablaXY}
  (\nabla_{Y_k}X_k)(x) := (\nabla_{Y}X)(x_k). 
\end{equation}
Since $Y_k(x)\in T_{x_k}\cM$, it must differentiate a tensor field at $x_k$,
so $X_k$ must be a function around $x_k$.
Hence, $X_k$ is $X$ when differentiated,
and then the entire result is pulled-back to $x$. 
The other way does not work: if $X_k$ is $X\circ f^k$ when differentiated, 
this is a function of $x$ and can not be differentiated by $Y_k$.
For another example, take a differentiable observable function $\phi$ on $\cM$,
\begin{equation} \begin{split} \label{e:notation}
  f^k_* Y(\phi_k)(x):= 
  f^k_* Y (\phi)(x_k) = Y (\phi\circ f^k)(x) =: Y(\phi_k)(x),
\end{split} \end{equation}
where $Y(\cdot)$ means to differentiate in the direction of $Y$.
Since $f^k_*Y(x)\in T_{x_k}\cM$,
the first $\phi_k$ is $\phi$ when the differentiation occurs at $x_k$,
and then the result is pulled-back to $x$;
since $Y(x)\in T_x\cM$,
the second $\phi_k$ is $\phi\circ f^k$ when the differentiation occurs right at $x$.
Finally, we omit the step subscript of the pushforward tensor $f_*$,
since its location is well-specified by the vector it applies to.

\subsection{Some previous linear response formulas} 
\label{s:linResponse}
\hfill\vspace{0.1in}

A linear response formula is an expression of $\delta \rho$, the derivative of the SRB measure, using $\delta f$.
The path-perturbation formula for the linear response is formally the average of perturbations of each orbit initially distributed as the SRB measure \cite{TrsfOprt}.
Ruelle proved that this formula indeed gives the derivative \cite{Ruelle_diff_maps,Jiang06}.
The formula can be proved for more general cases;
for example, Dolgopyat proved it for partially hyperbolic systems \cite{Dolgopyat2004}.

\begin{theorem}[path-perturbation formula for the linear response] \label{t:ruelle}
  Denote the SRB measure of $f$ on $K$ by $\rho$,
  assume that $f$ is parameterized by some scalar $\gamma$,
  and define 
  \[ \begin{split}
  \delta(\cdot):=\frac{\partial (\cdot)}  {\partial \gamma}.
  \end{split} \]
  Then the derivative of the SRB measure, for a fixed objective function $\Phi$, is given by:
  \[
    \delta \rho (\Phi)
    = \sum_{n=0}^\infty \rho\left( X(\Phi_n) \right) 
    = \lim_{W\rightarrow\infty} \rho\left( \sum_{n=0}^W X(\Phi_n) \right) ,
  \]
  where $X:=\delta f\circ f^{-1}$, $X(\cdot)$ is to differentiate in the direction of $X$,
  and $\Phi_n = \Phi\circ f^n$.
\end{theorem}

\begin{remark*}
(1)
The $f^{-1}$ in the definition of $X$ is to make $X(x)$ a vector at $x$:
this is more convenient than $\delta f$, which maps $x$ to a vector at $f(x)$.
(2)
The formula in the theorem is more important to us than the technical assumptions, which we choose to be the same as \cite{Ruelle_diff_maps}; our formula should be correct for more general assumptions, as discussed in \cref{s:beyond hyper}.

More technically, we assume that $K$ is a mixing Axiom A attractor for the $C^3$ diffeomorphism $f$ of $\cM$.
Let $\cA$ be the space of $C^3$ diffeomorphisms sufficiently close to $f$ in a fixed neighborhood $V$ of $K$; 
$\cA$ has $C^3$ topology.
Assume that $\gamma\mapsto f$ is $C^1$ from $\R$ to $\cA$, then $X$ is a $C^3$ vector field.
Fix an observable $\Phi\in C^2$, then $\gamma\mapsto \rho(\Phi)$ is $C^1$ from $\R$ to $\R$.
\end{remark*}

Due to unstable components, the path-perturbation estimator, where the estimator refers to the expression inside the integration, grows exponentially to $W$:
\[ \begin{split}
    \sum_{n=0}^W X(\Phi_n) \sim O(\lambda_{max}^{W}).
\end{split} \]
Here, $\lambda_{max} >1$ is the largest Lyapunov exponent.
This phenomenon is also known as the `gradients explosion'.
Statistically, the number of samples requested to evaluate the integration increases exponentially to $W$, incurring a high computational cost.
Hence, algorithms based on the path-perturbation formula suffer from very high computational cost in unstable systems.

To obtain a small estimator, we decompose the linear response into two parts.
The first part, the shadowing contribution, accounts for the change in the location of the attractor through a conjugacy map.
The second part, the unstable contribution, 
accounts for the fact that the pushforward of the old SRB onto the new attractor is no longer SRB.
Integrating by part the unstable contribution on the unstable manifold gives (see section~\ref{s:transform1} for explanations)
\begin{equation} \begin{split} \label{e:ruelle22}
\delta \rho (\Phi)  &= S.C. - U.C. 
  = \rho(v(\Phi)) -  
  \rho \left(\left( \sum_{n\in \mathbb{Z}}\Phi _n \right) \diverg_\sigma^u  X^u \right),
  \\ \quad\textnormal{where} \quad
  &S.C.
  := \sum_{n\ge0} \rho (X^s(\Phi_n))
  - \sum_{n\le -1} \rho (X^u(\Phi_n)) 
  = \rho(v(\Phi)),\\
  &U.C.
  := -\sum_{n\in \Z} \rho\left( X^u(\Phi_n) \right) 
  = \lim_{W\rightarrow \infty} \rho \left(\sum_{n=-W}^W \Phi _n  \diverg_\sigma^u  X^u\right) .\\
\end{split} \end{equation}
Here $S.C.$, $U.C.$ are shadowing and unstable contributions,
$v$ is the shadowing vector (section~\ref{s:shadowing contribution}),
$X^u$ and $X^s$ are unstable and stable oblique projections of $X$ (appendix~\ref{a:projection} figure~\ref{f:projection}).
The unstable divergence, $\diverg_\sigma^u$, is the divergence in the unstable manifold under the conditional SRB measure.

We call equation~\eqref{e:ruelle22} the blended linear response formula.
For the shadowing contribution, the integrand $v$ is bounded.
For the unstable contribution, note that 
subtracting $\Phi$ by any constant, such as $\rho(\Phi)$, does not change the linear response, so
\[ \begin{split}
  U.C.  = \lim_{W\rightarrow \infty} \rho \left(\psi  \diverg_\sigma^u  X^u\right),
  \quad\textnormal{where}\quad
  \psi:=\sum_{n=-W}^W (\Phi_n -\rho(\Phi)).
\end{split} \]
Note that $\psi$'s size is $O(\sqrt W)$ now.
Hence, the estimator in the blended formula is
much smaller than the path-perturbation formula,
and algorithms based on the blended formula should have much faster convergence.
However, efficient computation of the unstable divergence has been an open problem: this is solved in our paper.

We do not use the more familiar decomposition of the linear response into stable and unstable contributions.
This is because computing the stable contribution requires computing oblique projections,
which is twice the cost of computing the shadowing contribution by the nonintrusive shadowing algorithm; it is also not very robust.
Moreover, for both decompositions, our algorithm for the unstable contribution is the same,
which requires computing a modified shadowing vector anyway.

\subsection{Nonintrusive shadowing algorithm}
\label{s:shadowing contribution}
\hfill\vspace{0.1in}

On an orbit $ \{ x_n \}_{ n\in\Z }$, we define the shadowing vector, $ \{ v_n:=v(x_n)\in T_{x_n}\cM \}_{ n\in\Z }$, 
as the only bounded solution of the inhomogeneous tangent equation,
\[\begin{split}
  v'_{n+1} = f_*v'_n + X_{n+1}.
\end{split}\]
This equation governs the propagation of perturbations on an orbit,
hence, we know that $v$ is the first order difference between two shadowing orbits.
Notice that $v$ is a vector field and $ \{ v_n \}_{ n\in\Z }$ is defined on an orbit.
By the law of linear superposition and the exponential growth of homogeneous tangent solutions, 
we can write the formula of $v$,
\[ 
  v = \sum_{k\ge 0}f_*^k X_{-k}^s - \sum_{k \ge 1} f_*^{-k} X_{k}^u.
\]
Compared with the expression of the shadowing contribution, we can show
\[ 
  S.C.  = \rho (v( \Phi)), 
\]
Here, $v(\cdot)$ denotes the derivative in the direction of $v$.

The nonintrusive shadowing algorithm computes the shadowing vector by constrained minimization on an orbit \cite{Ni_NILSS_JCP},
\[ \begin{split} 
  &\min_{a\in \R^u} \frac 1{2N} \sum _{n=0} ^{N-1} \|v_n\|^2,
  \quad  \mbox{s.t. }
  v_n =  v_n' + \underline e_n a \,,
\end{split}\]
where $v'$ is an arbitrarily chosen inhomogeneous tangent solution, such as that solved from the zero initial condition.
$\underline e$ is an $M\times u$ matrix, whose columns are homogeneous tangent solutions.
In other words, the boundedness property is approximated by a minimization,
and the $M$-dimensional feasible space of all inhomogeneous tangent solutions is reduced to a $u$-dimensional affine subspace.
This reduction significantly reduces computational cost, yet is still capable of finding the shadowing vector \cite{Ruesha}.
The nonintrusive shadowing algorithm is an ingredient of the fast response algorithm,
and its pseudocode is included in section~\ref{s:procedure list}.

If the unstable contribution is small,
we may choose to ignore it or approximate it crudely.
For example, when $u\ll M$, the unstable contribution is small if the system has a fast decay of correlations, and $X$ and $\Phi$ are not particularly aligned with the unstable direction.
Then, we may well approximate the entire linear response by nonintrusive shadowing.
Furthermore, it is possible to partially compute the unstable contribution using the path-perturbation formula, but the asymptotic cost for large $W$ is much higher than fast response \cite{Ruesha}.
Since nonintrusive shadowing is part of the fast response algorithm, it does not hurt to first try nonintrusive shadowing, which is faster and maybe accurate enough.
If more accuracy is demanded, we may further compute the unstable contribution.
Somewhat surprisingly, an efficient computation of the unstable contribution also requires nonintrusive shadowing.

\section{Removing distributional derivatives and Expanding the unstable divergence} \label{s:3step}

The unstable divergence is a Holder continuous function, but the directional derivatives of $X^u$ are typically true distributions with infinite values.
Hence, we should search for new formulas for the unstable divergence other than summing the directional derivatives.
This is achieved via three steps of transformation, each corresponds to a subsection, and the latter two are perhaps more difficult.
\begin{itemize}
\item Transform the unstable divergence under the conditional SRB measure to the divergence under the Lebesgue measure.
The difference is the derivative of the conditional SRB measure, 
which can be expanded since SRB is the infinite pushforward of Lebesgue.
\item The Lebesgue unstable divergence equals the sum of the derivatives of the volume ratio of two projections, one along stable manifolds, the other along $X$.
\item Expand the projection along stable manifolds using the fact that it collapses after infinite pushforward.
\end{itemize}

\subsection{Integration by parts and measure change}
\label{s:transform1}
\hfill\vspace{0.1in}

In order to obtain a smaller integrand for the unstable contribution, we integrate-by-parts on the unstable manifold under the conditional SRB measure, which yields unstable divergence under the SRB measure.
We show that the SRB unstable divergence differs from the Lebesgue unstable divergence by a measure change.

Throughout this paper, we make the same assumptions as theorem~\ref{t:ruelle} for the simplicity of the discussions.
Recall that the conditional SRB measure on an unstable manifold $\cV^u$ has a density $\sigma$ with respect to the $u$-dimensional Lebesgue measure.
Let $\omega$ be the volume form on $\cV^u$, $\rho$ be the SRB measure.
Integrations to $\rho$ can be done by first integrating to $\sigma \omega$ on the unstable manifold, then in the transversal direction.

We first explain the integration by parts.
For any $\phi\in C^1$, 
\[ \begin{split} 
  X^u(\phi) \sigma  +  \phi X^u(\sigma)
  = X^u (\phi \sigma)
  = \diverg^u (\phi \sigma X^u) 
  -  \phi \sigma \diverg^u X^u .
\end{split} \]
Here $X^u(\cdot)$ means to differentiate a function in the direction of $X^u$, $\diverg^u$ is the divergence on the unstable manifold under the Riemannian metric.
For now, we only know that $X^u$ is Holder continuous and its derivatives are distributions; 
later, we will prove the regularity of the unstable divergence.
On a piece of unstable manifold $\cV^u$, 
we have the integration-by-parts formula
\begin{equation} \begin{split} \label{e:oxygen}
  \int_{\cV^u} X^u(\phi) \sigma \omega 
  = \int_{\cV^u} \diverg^u (\phi \sigma X^u) \omega
  - \int_{\cV^u} \left( \frac \phi \sigma X^u(\sigma)\right) \sigma \omega
  - \int_{\cV^u} \left( \phi \diverg^u X^u \right) \sigma \omega.
\end{split} \end{equation}
We will deal with the first two terms on the right hand side in this subsection,
and leave the last term, which involves the unstable divergence, to the next two subsections.

When integrating over the entire attractor with the SRB measure $\rho$, the first term on the right of equation~\eqref{e:oxygen} becomes zero.
To see this, first notice that the divergence theorem reduces it to boundary integrals.
Intuitively, since unstable manifolds always lie within the attractor and do not have boundaries,
the boundary integral would never appear when integrating over $\rho$, and hence this term becomes zero.
More rigorously, we can choose a Markov partition with a small diameter and then let the boundary integrals of two adjacent rectangles cancel each other.
To conclude, we get
\[ \begin{split} 
\rho( X^u(\phi) ) 
  = - \rho \left(\phi \diverg_\sigma ^uX^u\right),
  \textnormal{ where }
  \diverg_\sigma ^uX^u := \frac 1\sigma X^u(\sigma) + \diverg^u X^u .
\end{split} \]

In the rest of this paper, let 
\begin{equation} \begin{split} \label{e:e}
  e:=e_1\wedge\cdots\wedge e_u\in \wedge^u V^u 
\end{split} \end{equation}
be a $u$-vector field differentiable on each unstable manifold, but not necessarily continuous in all directions.
In addition, assume that $e$ and $\nabla_{(\cdot)} e$ are bounded on the attractor $K$ under the Riemannian metric, where $\nabla_{(\cdot)} e$ is the Riemannian connection operating on vectors in $V^u$.
In fact, we may regard $e$ as $C^1$ $u$-vector fields on individual unstable manifolds.
We can make it continuous across the foliation, for example,
\[ \begin{split}
  \tilde e := \frac {e}{\|e\|}
\end{split} \]
is unique and continuous modulo orientation, and it satisfies our boundedness assumption,
because the unstable manifold theorem states that, in our case,
the unstable manifolds are continuous in $C^3$ topology \cite{Pugh1977}.
Here $\|\cdot\|$ is the $u$-dimensional volume of the hyper-cube spanned by $\{e_i\}_{i=1}^u$.
We use $e$ instead of $\tilde e$ if the statement holds more generally.

The volume of hyper-cubes, $\|\cdot\|$, is in fact a tensor norm induced by the Riemannian metric.
That is, $\|e\|^2 = \ip{e,e}$.
In this paper, for simple $u$-vectors,
\begin{equation} \begin{split} \label{e:tensorproduct}
  \ip{e,r} := \det \ip{e_i, r_j},
  \quad \textnormal{where} \quad
  e = e_1\wcw e_u, \;
  r = r_1\wcw r_u, \;
  e_i, r_j\in T\cM.
\end{split} \end{equation}
When the operands are summations of simple $u$-vectors, the inner-product is the corresponding sum \cite{chenweihuan,LeeRieman}.

The SRB measure is the weak limit of pushing-forward the Lebesgue measure.
Roughly speaking, pushing-forward is like a matrix multiplication.
Hence, by the Leibniz rule,
the measure change term, $X^u(\sigma)/\sigma$, can be expanded to an infinite summation given below.

\begin{lemma}[expression for measure change] \label{l:arau}
The measure change has the following expression that converges uniformly on $K$
\[ \begin{split} 
  \frac 1\sigma X^u(\sigma)
  = \sum_{k=1}^\infty
  - \frac{\ip{\nabla_{f_*^{-k+1}X^u} f_*e_{-k}, f_*e_{-k}} }
  {\ip{f_*e_{-k}, f_*e_{-k}}}
  + \frac{\ip{\nabla_{f_*^{-k}X^u} e_{-k}, e_{-k}}} 
  {\ip{e_{-k}, e_{-k}}} ,
\end{split} \]
where $e_{-k}$ can be any differentiable $u$-vector field on $f^{-k}(\cV^u)$.
\end{lemma}

\begin{remark*}
  (1) If we evaluate both sides of the equation at $x$, then $e_{-k}= e(f^{-k}x)$, and when being differentiated, $e_{-k}$ is a restriction of $e$ to a neighborhood of $x_{-k}=f^{-k}x$.
  (2) Due to uniform convergence, $ X^u(\sigma) /\sigma$ is uniform continuous over $K$.
  (3) An algorithm computing this expansion would converge to the true solution.
  (4) We can show that this term is Holder-continuous using the basic arguments in the appendix of \cite{Ni_asl}.
\end{remark*}

\begin{proof}
The conditional SRB measure $\sigma$ on $\cV^u$ is the result of
evolving the Lebesgue measure starting from the infinite past.
More specifically, 
\[ \begin{split}
  \sigma
  = \prod _{k=1} ^\infty C_k \left(J^u_{-k} \right)^{-1},
\end{split} \]
where $C_k$ is constant over $\cV^u_k$ to keep the total conditional measure at $1$,
and $J^u$ is the unstable Jacobian computed with respect to the Riemannian metric on $\cV^u$,
more specifically, for any $e$,
\begin{equation} \begin{split} \label{e:hotel}
  J^u
  := \frac{\| f_* e\|}{\|e\|}
  = \left(\frac{\ip{ f_* e, f_* e} }{\ip{e, e}}\right)^{0.5}.
\end{split} \end{equation}
Notice $J^u$ does not depend on the particular choice of $e$.
By the Leibniz rule of differentiation,
and the notation convention explained in equation~\eqref{e:notation},
\begin{equation} \begin{split} \label{e:included}
  \frac 1\sigma X^u(\sigma)
  = - \sum_{k=1} ^\infty \frac {X^u (J^u_{-k})}{ J^u_{-k} }
  = - \sum_{k=1} ^\infty \frac {(f^{-k}_* X^u) J^u_{-k}}{ J^u_{-k}}.
\end{split} \end{equation}

To get the equation in the lemma, substitute~\eqref{e:hotel} into~\eqref{e:included}.
For any vector $Y\in V^u$,
\[ \begin{split}
  &Y(J^u) 
  = Y\left( \frac{\ip{f_*e, f_*e} }{ \ip{e, e}} \right)^{0.5}
  = \frac 1 2 (J^u)^{-1} Y\left( \frac{\ip{f_*e, f_*e}}{ \ip{e, e}} \right) \\
  =& \frac 1 2 (J^u)^{-1} 
  \frac{\ip{f_*e, f_*e}}{ \ip{e, e}}
  \left[ \frac{f_*Y\ip{f_*e, f_*e} }{\ip{f_*e, f_*e}} 
  - \frac{Y\ip{e, e}}{\ip{e, e}} \right]
  = J^u 
  \left[ \frac{\ip{\nabla_{f_*Y}f_*e, f_*e} }{\ip{f_*e, f_*e}} 
  - \frac{\ip{\nabla_Y e, e}}{\ip{e, e}} \right].
\end{split} \]
In the last equality, we applied the rule for differentiating the Riemannian metric.

To see uniform convergence, take $e$ as the vector field such that $e$ and its derivative in the unit unstable direction are bounded on $K$, then the series is controlled by the exponentially shrinking term, $\|f_*^{-k} X^u\| \le C \lambda^{-k}\|X^u\|$.
\end{proof}

\subsection{Transforming to the derivative of volume ratio}
\hfill\vspace{0.1in}
\label{s:transform2pi}

Computing $\diverg ^u X^u$ is difficult since $X^u$ is typically not differentiable even within an unstable manifold, so we can not compute the pointwise value of $\div^u X^u$ as the sum of directional derivatives.
However, $\diverg^uX^u$ can be more regular than typical distributions, and Ruelle proved its Holder continuity in \cite{Ruelle_diff_maps_erratum}.
This extra regularity is not too surprising, because (1) the summation has some regularizing effects, and (2) the divergence should be more naturally interpreted as the derivative of transfer operators on measures, rather than the sum of directional derivatives \cite{TrsfOprt}.

The main goal of this subsection is to transform $\diverg^uX^u$ to the derivative of $\varpi$, the volume ratio of the holonomy map along the stable direction.
This gets rid of directional derivatives, which are distributions.
To achieve this, we first define $\varpi$, then give a quick and formal derivation of the new expression for $\diverg^uX^u$, and then prove it rigorously.

\subsubsection{Definitions of \texorpdfstring{$\varpi$}{pi} and related quantities}
\hfill\vspace{0.1in}
\label{s:define varpi}

As illustrated in figure~\ref{f:local coordinate},
fix an unstable manifold $\cV^u$, let $q\in\R$ be a small parameter, for any $y\in \cV^u$,
define $\eta^q(y):\cV^u\rightarrow \cM$ as the unique curve such that
$\partial \eta^q / \partial q= X$ and $\eta^0(y) = y$.
For fixed $q$, $\cV^{uq}:=\{\eta^q(y):y\in\cV^u\}$ is a $u$-dimensional $C^3$ manifold;
for a small interval of $q$, $\hat \cV^u := \cup_q \cV^{uq}$ is a $u+1$ dimensional manifold.
For any $y\in\cV^u$, 
denote the stable manifold that goes through it by $\cV^s(y)$,
which is not differentiable with respect to $y$ in $C^3$.
Define $\xi^q(y)$ as the unique intersection point of $\cV^s(y)$ and $\cV^{uq}$.

Define $\pi_\eta, \pi_\xi: \hat\cV^u \rightarrow \cV^u$ such that
$\pi_\eta(x)=(\eta^q)^{-1}x$, $\pi_\xi(x) = (\xi^q)^{-1}x$,
for any $x\in \cV^{uq}$ and small $q\in \R$.
Denote $\eta^q_*, \xi^q_*$ as the pushforward operator of $\eta^q, \xi^q$.
For any small $q\in\R$, define 
\[ \begin{split}
  \eta_*e(x):= \eta^q_*e(\pi_\eta(x)), \quad
  \xi_* e := \xi^q_*e(\pi_\xi(x)), \quad
  \textnormal{for } x\in \cV^{uq}.
\end{split} \]
Then $\eta_*e$ and $\xi_*e$ are two parallel $u$-vector fields on $\hat\cV^u$.
(We may equivalently define $\eta_*e:=\cup_q \eta^q_*e$, $\xi_* e:=\cup_q \xi^q_*e$.)
Define a function $\varpi$ on $\hat\cV^u$ as the volume ratio,
\begin{equation} \label{e:define pi}
  \frac{\eta_*e}{\|e\circ \pi_\eta \|}
  = \varpi 
  \frac{\xi_*e}{\|e\circ \pi_\xi \|}.
\end{equation}
By transversal absolute continuity, $\varpi$ is a well-defined measurable function.
We can view $\varpi$ as a function of $x$, or a function of $(y_1, q)$, or a function of $(y_2, q)$.

\begin{figure}[ht] \centering
  \includegraphics[width=0.6\textwidth]{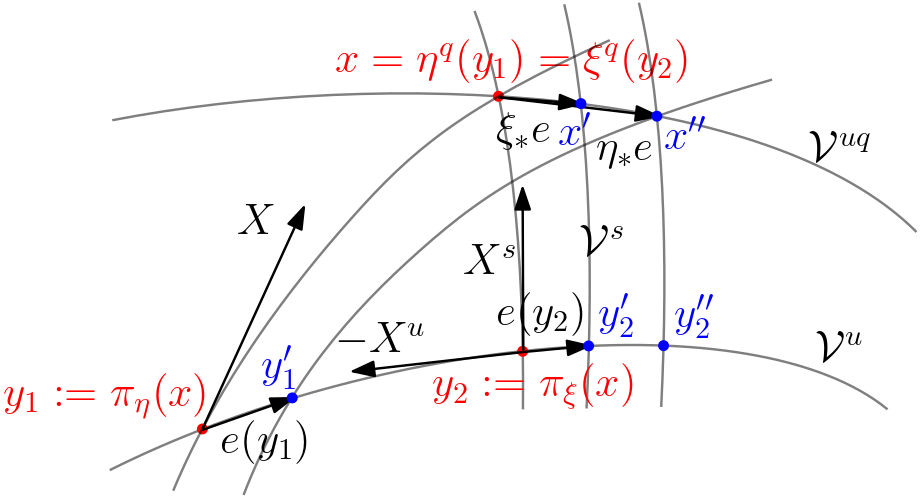}
  \caption{Definitions of projections.}
  \label{f:local coordinate}
\end{figure}

\subsubsection{Intuitive proof}
\hfill\vspace{0.1in}
\label{s:intuitive varpi}

We give an intuitive but non-rigorous proof of a simplified version of \cref{l:hen} in \cref{s:rigorous varpi}, the main lemma of \cref{s:transform2pi}, that is,
\begin{equation*}\begin{split}
  1 + q\, \div^u X^u = \varpi(x) + O(q^2).
\end{split}\end{equation*}
To see this, assume that $\cM=\R^2$ and unstable dimension $u=1$, so we can roughly express the quantities in \cref{f:local coordinate} as
\begin{equation*}\begin{split}
  e(y_1) = y_1'-y_1,
  \quad \textnormal{} \quad 
  e(y_2) = y_2'-y_2,
  \\
  \xi_*e (x) = x' - x,
  \quad \textnormal{} \quad 
  \eta_*e (x) = x'' - x.
\end{split}\end{equation*}
We may assume that all the above quantities have size $O(1)$, since we only care about their quotients.
With these, we can express the two quantities of interest.

By the definition of $\varpi$ in \cref{e:define pi}, we have 
\begin{equation*}\begin{split}
  \varpi(x)  
  :=\varpi(y_1, q)  
  := \frac{\|\eta_*e\|}{\|\xi_*e\|}
  \frac{\|e\circ \pi_\xi \|}{\|e\circ \pi_\eta \|}
  = \frac{\|x''-x\|}{\|x'-x\|} \frac{\|e(y_2)\|}{\|e(y_1)\|}
\end{split}\end{equation*}
We may assume that there is a Lipschitz-continuous function $k$ such that
\begin{equation*}\begin{split}
  \|x'-x\| = \|y_2'-y_2\| (1+k(y_2, y_2')q),
  \quad \textnormal{} \quad 
  \|x''-x\| = \|y_2''-y_2\| (1+k(y_2, y_2'')q)
\end{split}\end{equation*}
Since $y_2''-y_2'=O(q)$, we have
\begin{equation*}\begin{split}
  \frac{\|x''-x\|}{\|x'-x\|} 
  = \frac{\|y_2''-y_2\|}{\|y_2'-y_2\|} 
  \,\frac{1+k(y_2, y_2'')q}{1+k(y_2, y_2')q} 
  = \frac{\|y_2''-y_2\|}{\|y_2'-y_2\|} +O(q^2).
\end{split}\end{equation*}
Hence,
\begin{equation*}\begin{split}
  \varpi(x)
  = \frac{\|y_2''-y_2\|}{\|y_2'-y_2\|} \frac{\|e(y_2)\|}{\|e(y_1)\|}  + O(q^2)
  = \frac{\|y_2''-y_2\|}{\|e(y_1)\|}  + O(q^2).
\end{split}\end{equation*}

For $\div^uX^u$, note that $y_2- y_1$ is the distance along the direction of $X^u(y_1)$ for a duration of $q$, $y_2''- y_1'$ is the distance along the direction of $X^u(y_1')$ for a duration of $q$, so, roughly speaking, we have
\begin{equation*}\begin{split}
  \div^uX^u (y_1) 
  := \frac{\|(y_2''- y_1')- (y_2-y_1)\|}{\|e(y_1)\| q} + O(q)
  = \frac{\|(y_2''- y_2) - (y_1'-y_1)\|}{\|e(y_1)\| q} + O(q).
\end{split}\end{equation*}
The last equality transfers the derivative of $X^u$ along unstable directions into a derivative of a volume ratio.
This is the pivotal step where we get rid of the non-differentiability of oblique projections.
It corresponds to \cref{l:devil}, which is the rigorous geometric version of essentially the same trick.

Hence, we obtain a small-perturbation version of \cref{l:hen},
\begin{equation*}\begin{split}
  1 + q\, \div^u X^u
  = 1 + \frac{\|(y_2''- y_2) - (y_1'-y_1)\|}{\|e(y_1)\|}+ O(q^2)
  \\
  = 1 + \frac{\|(y_2''- y_2)\| - \|(y_1'-y_1)\|}{\|e(y_1)\|}+ O(q^2)
  = \frac{\|(y_2''- y_2)\|}{\|e(y_1)\|} + O(q^2)
  = \varpi(x) + O(q^2).
\end{split}\end{equation*}

\subsubsection{Rigorous proof}
\hfill\vspace{0.1in}
\label{s:rigorous varpi}

\begin{lemma}[expression of $\nabla_{X^s} \eta_* e$] \label{l:Richter}
Denote $\nabla_e X:=\sum_i e_1\wedge\cdots \wedge \nabla_{e_i}X\wedge\cdots\wedge e_u$.
On $\cV^u$,
  \[ \begin{split}
    \nabla_{X^s}(\eta_* e_i) 
    = \nabla_{e_i} X - \nabla_{X^u} e_i \,,\quad
    \nabla_{X^s}(\eta_* e) 
    = \nabla_{e} X - \nabla_{X^u} e.
  \end{split} \]
\end{lemma}

\begin{remark*}
  $\nabla_{e} X$ is a function, not a distribution, because $X$ is differentiable on $\cM$.
  $\nabla_{X^u} e$ also is a function, 
  since it only requires that $e$ be differentiable along the direction of $X^u$.
 The differentiability of $X^u$ is not required and should not be required, as $X^u$ is not differentiable.
  \end{remark*}

\begin{proof}
With our notation, the proof below works for both $e$ and $e_i$.
First decompose,
\[ \begin{split}
  \nabla_{X^s}(\eta_* e) 
  = \nabla_{X}(\eta_* e) - \nabla_{X^u}(\eta_* e) 
\end{split} \]
Since $\eta$ is the flow of $X$, the Lie derivative
$L_{X}(\eta_* e) = 0$, hence, on $\cV^u$,
\[ \begin{split}
  \nabla_{X}(\eta_* e) = \nabla_{\eta_* e}X = \nabla_{e} X,
\end{split} \]
where $\eta_*e = e$ on $\cV^u$.
Since $X^u$ is a vector field on $\cV^u$,
$\nabla_{X^u}(\eta_* e)  = \nabla_{X^u}e$.
\end{proof}

\begin{lemma} \label{l:devil}
  Further assume that $\cV^s(y)$ varies smoothly with $y$, then on $\cV^u$,
  \[
    \nabla_{X^s}(\xi_* e_i) 
    = \nabla_{e_i} X^s, \quad
    \nabla_{X^s}(\xi_* e) 
    = \nabla_{e} X^s.
  \]
\end{lemma}

\begin{remark*}
  (1)
  Intuitively, this lemma just says that $(x'-x) - (y_2'-y_2) = (x'-y_2') - (x-y_2)$ in \cref{f:local coordinate}.
  (2)
  The extra smoothness assumption is temporary; it makes $\xi_*e$ a differentiable $u$-vector field on $\cV^{uq}$.
  Without this assumption, we may temporarily interpret this lemma in the sense of distributions, and later we can see that $\xi_* e$ is still differentiable in the direction of $X^s$, by the expansion formula in section~\ref{s:transform3}.
\end{remark*}

\begin{proof}
In $\hat \cV^u$, $\partial \xi^q/\partial q$ is a smooth vector field.
  We claim that $\partial \xi^q/\partial q = X^s$ on $\cV^u$.
  To see this, define $y_1^q:\cV^u\rightarrow \cV^u$ as
  \[ \begin{split}
    y_1^q(y_2) = \pi_\eta( \xi^q(y_2)) ,
    \quad \textnormal{or equivalently,}\quad
    \xi^q(y_2) = \eta^q(y_1^q(y_2)) .
  \end{split} \]
  Fix $y_2$, differentiate the second equation to $q$, 
  evaluate at $q=0$, where $\partial\eta^q/\partial y_1 = Id$,
  \[ \begin{split}
  \pp {\xi^q} q = \pp {\eta^q} {y_1} \pp {y_1^q}q + \pp {\eta^q} q = \pp {y_1^q}q + X.
  \end{split} \]
  Since $\partial y_1^q/\partial q\in V^u$, $\partial \xi^q /\partial q\in V^s$,
  and that $X$ is uniquely decomposed in $V^u \bigoplus V^s$ into $X=X^u\oplus X^s$,
  we see that 
  \begin{equation} \begin{split} \label{e:mozart}
    \pp {\xi^q} q = X^s  ,\quad
    \pp {y_1^q} q = -X^u  
    \quad\textnormal{at}\quad
    q=0    .
  \end{split} \end{equation}
  Hence, on $\cV^u$,
  \[ \begin{split}
    \nabla_{X^s}(\xi_* e) 
    = \nabla_{\partial \xi^q/\partial q}(\xi_* e) 
    = \nabla_{\xi_* e} {\partial \xi^q/\partial q}
    = \nabla_{e} X^s,
  \end{split} \]
  where the second equality is due to the Lie derivative $L_{\partial \xi^q/\partial q}(\xi_* e)=0$
  by definitions,
  and the last equality uses $\xi_*e =e$ at $q=0$.
\end{proof}

\begin{lemma}\label{l:hen}
  Under the same assumption as lemma~\ref{l:Richter},
  $ X^s(\varpi)= \diverg^u X^u$ on $\cV^u$.
\end{lemma}

\begin{remark*}
To show that this equivalence persists in general cases where stable foliation is not smooth, first establish the lemma on a smooth foliation and then evolve backward in time to approach the stable foliation.
\end{remark*}

\begin{proof}
By definitions,
  \[ \begin{split}
      y_1^q(x) = \pi_\eta(\xi^q(x_1))
      ,\quad
      y_2 = \pi_\xi(\xi^q(y_2)).
  \end{split} \]
Differentiating to $q$, and using \cref{e:mozart}, we get
 \begin{equation} \begin{split} \label{e:beifeng}
    -X^u = \pp {y_1^q} q
    = \pi_{\eta*} \pp{\xi^q}q
    = \pi_{\eta*} X^s
    ,\quad 
    0= \pi_{\xi*} \pp{\xi^q}q 
    = \pi_{\xi*} X^s 
    \quad \textnormal{on} \quad \cV^u.
  \end{split} \end{equation}
Recall that ${X^s} \| e\circ\pi_\eta \|^2
    = (\pi_{\eta*} X^s) \| e \|^2$ by \cref{e:notation}, we get
  \[ \begin{split}
    {X^s} \| e\circ\pi_\eta \|^2
    = (\pi_{\eta*} X^s) \| e \|^2
    = \nabla_{-X^u} \| e \|^2
    = 2 \ip{\nabla_{-X^u} e, e}, \quad
    X^s \| e\circ\pi_\xi \|^2 =0.
  \end{split} \]

Take the inner product of each side of equation~\eqref{e:define pi} with itself, 
  \[
    \left\| \eta_* e \right\|^2 \| e\circ\pi_\xi \|^2
    = \varpi^2 \| \xi_* e\|^2 \| e\circ\pi_\eta \|^2.
  \]
Differentiating in the direction of $X^s$ and evaluating on $\cV^u$, applying \cref{e:beifeng}, noting that $\varpi = 1$, $\eta_*=\xi_*=Id$ on $\cV^u$, 
  \[ \begin{split}
    \ip{\nabla_{X^s}(\eta_* e), e} 
    = X^s(\varpi) \ip{ e,  e} 
    +  \ip{\nabla_{X^s}(\xi_* e), e}
    +  \ip{\nabla_{-X^u}e,e}.
  \end{split}\]
  By lemma~\ref{l:Richter} and~\ref{l:devil}, we have
  \[ \begin{split}
    X^s(\varpi)
    = \frac 1 {\ip{e,e}} 
      \ip{\nabla_{e} X - \nabla_{X^u} e - \nabla_{e} X^s + \nabla_{X^u} e, e}
    = \frac 1 {\ip{e,e}} \ip{\nabla_{e} X^u, e}.
  \end{split} \]

  To see that this is a divergence, recall that the Riemannian connection within a submanifold $\cV^u$ is the orthogonal projection 
  of that on the background manifold $M$, so
  \[ \begin{split}
    X^s(\varpi) 
    = \frac 1 {\ip{e,e}} \ip{\nabla_{e}^u X^u, e}_u
    = \frac 1 {\ip{e,e}} 
    \sum_{i=1}^u \ip{e_1\wedge\cdots \wedge 
    \sum_j e^j_u (\nabla_{e_i}^u X^u) e_j 
    \wedge\cdots \wedge e_u, e}_u.
  \end{split} \]
Here $\nabla^u$, $\ip{\cdot,\cdot}_u$, and $\{e^j_u\}_{j=1}^u$ are the Riemannian connection, metric, and dual basis of $\{e_j\}_{j=1}^u$ within $\cV^u$.
The terms with $j\ne i$ vanish because the same direction appears twice in the exterior product, hence
  \[ \begin{split}
    X^s(\varpi) = 
    \sum_{i=1}^u e^i_u (\nabla_{e_i}^u X^u).
  \end{split} \]
  This is a contraction of $\nabla^u X^u$ within $\cV^u$, 
  which is the definition of $\diverg^u X^u$.
\end{proof}

\subsection{Expanding the volume ratio} \label{s:transform3}
\hfill\vspace{0.1in}

We show that $X^s(\varpi)$ is a continuous function by expanding the volume ratio into a converging summation.
The main problem now is that $\xi_*$ seems non-differentiable.
But $\xi_*$ is in fact a so-called holonomy map that is differentiable along the projection direction.
To prove the differentiability and write the derivative, we further write $\xi_*$ as an infinite pushfoward, whose derivative is an infinite summation.
This finally expands the derivative of the volume ratio and hence Lebesgue unstable divergence.

\begin{lemma}[expansion of Lebesgue unstable divergence]\label{l:hao}
  $X^s(\varpi)$ has the following expansion formula that converges uniformly on $K$
  \[ \begin{split} 
    \diverg^uX^u
    = X^s(\varpi)
    = \frac{\ip{\nabla_{e}X, e}}{\|e\|^2} 
    + \sum_{k=0}^\infty
    \frac {\ip{\nabla_{f_*^{k+1}X^s}f_*^{k+1}\eta_* e,f_*^{k+1} e}}
    {\|f_*^{k+1} e\|^2}
    -\frac{\ip{\nabla_{f_*^{k}X^s}f_*^{k}\eta_* e,f_*^{k} e}}
    {\|f_*^k e\|^2} .
  \end{split} \]
\end{lemma}

\begin{remark*}
Due to uniform convergence, $\diverg^uX^u$ is continuous over $K$.
A more careful analysis would show that it is Holder continuous over $K$ (see \cite{Ni_asl} for example).
\end{remark*}

\begin{proof}
By definition of the stable manifold, 
for $x\in\hat \cV^u$,
$\lim_{k\rightarrow \infty} f^k(\pi_\xi x) = \lim_{k\rightarrow \infty} f^k(x)$.
Hence, 
\[ \begin{split} 
  \lim_{k\rightarrow \infty} 
  {\|f_*^k\xi_* e\|^2} / {\|f_*^k (e\circ\pi_\xi) \|^2} = 1.
\end{split} \]
Because $\xi_*e$ is parallel to $\eta_*e$, 
\[ \begin{split}
  \frac {\|\eta_*e\|} {\|\xi_*e\|}
  = \lim_{k\rightarrow \infty} \frac {\|f_*^k \eta_*e\|} {\|f_*^k\xi_*e\|}
  = \lim_{k\rightarrow \infty} \frac {\|f_*^k \eta_*e\|} {\|f_*^k(e\circ \pi_\xi)\|}
\end{split} \]
We can use this to replace $\xi_*e$ in the definition of $\varpi$ 
in equation~\eqref{e:define pi}, which yields an infinite pushforward:
\[ \begin{split}
  \varpi^2 
  = \frac{\|e\circ\pi_\xi\|^2} {\|e\circ\pi_\eta\|^2}
    \lim_{k\rightarrow \infty} 
    \frac {\|f_*^k\eta_* e\|^2}{\|f_*^k (e\circ\pi_\xi) \|^2} 
  = \frac {\|\eta_* e \|^2}{\|e\circ \pi_\eta \|^2}
    \prod_{k=0}^\infty
    \frac {\|f_*^{k+1}\eta_* e\|^2}{\|f_*^{k+1} e\circ\pi_\xi\|^2}
    \frac {\|f_*^k e\circ \pi_\xi \|^2} {\|f_*^k\eta_* e\|^2}.
\end{split} \]
Further differentiating in direction $X^s$ would yield an infinite summation.
More specifically, notice that at $\cV^u$, $\eta_*e = e\circ\pi_\xi=e$, $\varpi=1$,
and apply equation~\eqref{e:notation},
\[ \begin{split}
  &X^s(\varpi)
  = \frac{\ip{\nabla_{X^s} \eta_* e, e}}{\|e\|^2} 
  - \frac{ \ip{\nabla_{-X^u} e, e} }{\|e\|^2}
  + \sum_{k=0}^\infty
  \frac 
  {\ip{\nabla_{f_*^{k+1}X^s}f_*^{k+1}\eta_* e,f_*^{k+1} e}}
  {\|f_*^{k+1} e\|^2}
  -\frac 
  {\ip{\nabla_{f_*^{k}X^s}f_*^{k}\eta_* e,f_*^{k} e}}
  {\|f_*^k e\|^2} .
\end{split} \]
The expansion in this lemma is then proved by applying \cref{l:Richter,l:hen}.

Then we prove uniform convergence by estimating the terms in the expansion.
Using the Leibniz rule in appendix~\ref{a:f*} lemma~\ref{l:Leibniz}, we have
\[ \begin{split}
\nabla_{f_*^{k+1}X^s}f_*^{k+1}\eta_* e
= (\nabla_{f_*^k X^s}f_*) f_*^{k} \eta_* e
+ f_* \nabla_{f_*^k X^s} f_*^{k} \eta_* e.
\end{split} \]
Here $\nabla f_*$ is the Riemannian connection of $f_*$ (see appendix~\ref{a:f*} definition~\ref{d:dnabla}).
Note that $\eta_*$ is the identity on $\cV^u$.

Then using the projection operators defined in appendix~\ref{a:projection}, we can decompose the $k$-th term in expansion, $S_k$,
into $S_k=S_{k1} + S_{k2} +S_{k3}$, where
\[ \begin{split}
  S_{k1} := 
    \frac 1
    {\|f_*^{k+1} e\|^2}
    {\ip{(\nabla_{f_*^k X^s}f_*) f_*^{k}e, f_*^{k+1} e}}
  \le 
    C u \lambda^k \|X^s\|.
\end{split} \]
This inequality is straightforward when $u=1$,
because $X^s$ decays exponentially through pushforwards and $\nabla f_*$ is bounded.
For general $u$, recall that we can freely select $e_i$'s as long as their wedge product is $e$;
now let $ \{ f^k_*e_i \}_{ i=1 }^u$ be orthogonal at $f^k x$, hence
\[ \begin{split}
  \frac {\|(\nabla_{f_*^k X^s}f_*) f_*^{k}e_1 \wedge f_*^{k+1}e_2 \wcw f_*^{k+1} e_u\|}
  {\|f_*^{k+1} e\|}
  \le  
  \frac {\|(\nabla_{f_*^k X^s}f_*) f_*^{k}e_1 \|}
  {\|f_*^{k+1} e_1\|}
  \frac 
  {\|f_*^{k+1}e_2 \wcw f_*^{k+1} e_u \|}
  {\|f_*^{k+1}e_2 \wcw f_*^{k+1} e_u \|},
\end{split} \]
which reduces to the case $u=1$.

The term $S_{k2}$ is defined as
\[ \begin{split}
  S_{k2} :=
    \frac 1
    {\|f_*^{k+1} e\|^2}
    {\ip{f_*P^u\nabla_{f_*^{k}X^s}f_*^{k}\eta_* e,f_*^{k+1} e}}
    -\frac 1
    {\|f_*^k e\|^2} 
    {\ip{P^u\nabla_{f_*^{k}X^s}f_*^{k}\eta_* e,f_*^{k} e}}
  =0,
\end{split} \]
This is because $(V^u)^{\wedge u}$ is 1-dimensional, 
so $P^u\nabla_{\pp{}q}f_*^{k}\eta_* e(y)$ and $f_*^k e$ increase by same amounts via the pushforward.

Finally, $S_{k3}$ is defined as
\[ \begin{split}
  S_{k3} :=
    \frac 1
    {\|f_*^{k+1} e\|^2}
    {\ip{f_*P^s\nabla_{f_*^{k}X^s}f_*^{k}\eta_* e,f_*^{k+1} e}}
    -\frac 1
    {\|f_*^k e\|^2} 
    {\ip{P^s\nabla_{f_*^{k}X^s}f_*^{k}\eta_* e,f_*^{k} e}}.
\end{split} \]
The sum $\sum_{k\ge0}S_{k3}$ also converges uniformly on $K$.
To see this, first expand 
\begin{equation*}
    \nabla_{f_*^{k}X^s}f_*^{k}\eta_* e 
    = f_*^k \nabla_{X^s}\eta_* e
    + \sum_{n=0}^{k-1} f_*^{k-n-1}(\nabla_{f_*^nX^s} f_*) f_*^n e,
\end{equation*}
then apply \cref{l:Richter}, triangular equality, and \cref{l:P with f}, to get
\begin{equation} \begin{split} \label{e:mechanical}
  \frac {\|P^s\nabla_{f_*^{k}X^s}f_*^{k}\eta_* e\|} {\|f_*^k e\|} 
    \le \frac  {\|f_*^k P^s ( \nabla_{e} X - \nabla_{X^u} e)\|} {\|f_*^k e\|}
    + \sum_{n=0}^{k-1} \frac {\|f_*^{k-n-1} P^s (\nabla_{f_*^nX^s} f_*) f_*^n e \|}
    {\|f_*^{k-n} f_*^n e\|}\\
    \le C u \lambda^{2k} \|X\| 
    + C u \sum_{n=0}^{k-1}\lambda ^{2k-2n-1} \frac{\|f_*^nX^s\| \|f_*^n e\|}{\|f_*^n e\|}
    \le \frac{C u \lambda^k \|X\|}{1-\lambda}.\\
\end{split} \end{equation}
Here, the second inequality is again straightforward when $u=1$;
for general $u$, let $ \{ f^k_*e_i \}_{ i=1 }^u$ be orthogonal at $f^k x$,
and use the same trick for $S_{k1}$.

Hence, $\sum_{k\ge0} S_k$ uniformly converges.
\end{proof}

\begin{theorem}[expansion of unstable divergence] \label{t:expansion}
  Define 
  \[ \begin{split}
    \psi:=\sum_{m=-W}^W (\Phi_m -\rho(\Phi)),\quad
    \Psi:= \psi X.
  \end{split} \]
  By \cref{l:arau,l:hao},
  the unstable contribution is $U.C.= \lim_{W\rightarrow \infty} U.C.^W $, where
  \[ \begin{split}
    & U.C.^W
    := 
    \rho (\psi \diverg_\sigma^uX^u)
    = \rho \left( \psi \diverg^uX^u + \frac\psi\sigma X^u(\sigma)  \right) \\
    =& \rho \Bigg[ 
      \frac{\ip{\psi \nabla_{e}X, e}}{\|e\|^2} 
    + \sum_{k=0}^\infty \left(
      \frac {\ip{\nabla_{f_*^{k+1} \Psi^s}f_*^{k+1}\eta_* e, f_*^{k+1} e}}
      {\|f_*^{k+1} e\|^2}
    -\frac{\ip{\nabla_{f_*^{k} \Psi^s} f_*^{k} \eta_* e, f_*^{k} e}}
      {\|f_*^k e\|^2} \right)
    \\&
    - \sum_{k=1}^\infty \left(
      \frac{\ip{\nabla_{f_*^{-k+1} \Psi^u}f_* e_{-k}, f_* e_{-k}}}
      {\|f_* e_{-k}\|^2}
    - \frac{\ip{\nabla_{f_*^{-k} \Psi^u} e_{-k}, e_{-k}}}
      {\|e_{-k}\|^2} \right)
    \Bigg].
  \end{split} \]
  Here $\sigma$ is the density of the conditional SRB measure,
  $e_{-k}$ is a $u$-vector field on $f^{-k}\cV^u$ as defined in equation~\eqref{e:e},
  and $(\cdot)^u, (\cdot)^s$ are unstable and stable projections of a vector.
\end{theorem}

\begin{remark*}
(1)
The uniform convergence in lemma~\ref{l:arau} and~\ref{l:hao} shows that this formula
also converges uniformly.
(2)
Note that adding a constant to $\Phi$ does not change the linear response but helps reduce numerical errors.
\end{remark*}

Since the estimator is much smaller than the path-perturbation formula,
directly computing the expansion formula in theorem~\ref{t:expansion}
would have much faster convergence than algorithms based on the path-perturbation formula.
However, that would require solving at least $u^2$ second-order tangent equations, which will be defined later.
Directly computing the expansion formula would also require the oblique projections, 
which can be computed via a `little-intrusive' algorithm, which uses not only the tangent solvers, but also the adjoint solvers \cite{Ni_adjoint_shadowing,Ruesha}.
In the next section, we further transform the expansion formula to a `fast' formula,
which requires only $u$ many second-order tangent solutions and does not involve oblique projections.

\section{Fast formula of unstable divergence}
\label{s:compute unstable}

Fast algorithms, such as the fast Fourier transformation \cite{FFT} and the fast multipole method \cite{FMM,Greengard1991},
found a non-obvious `fast' structure, which allows us to combine many small terms into a few big terms.
Because the rules of propagation, or the `dynamics', on these small terms are similar, we can then apply some averaged rule of propagation on big terms only a few times, instead of many times on each small term.

In this section, we show that the expansion formula of the unstable contribution can be expressed via this `first combine, then propagate' strategy.
Then we show that the main term is the unique limit of the dynamics it satisfies.
It also turns out that the oblique projection disappears in the fast formula.
Hence, the fast formula involves only $u$ many first- and second-order tangent solutions.

\subsection{Definitions of \texorpdfstring{$p$, $\beta$, and $U$}{p, beta, and U}} 
\label{s:betaU}
\hfill\vspace{0.1in}

We shall find a very fast structure for the expansion formula of the unstable divergence, where we only need one uniform rule of propagation for one big term.
This big term is $p$, which is the linear sum of many small terms, $\nabla e$'s.
The dynamics of $p$ and $\nabla e$ is denoted by $\beta$, which is the renormalized second-order tangent equation, which governs the propagation of derivatives of vectors.
Finally, $U$ accounts for the volume growth along the unstable direction after applying $\beta$, and the unstable contribution can be expressed as $\rho(\tilde U (p))$.

More specifically, the expression in theorem~\ref{t:expansion} computes $\nabla f_*$ too many times,
which is the Riemannian connection of the pushforward tensor defined in appendix~\ref{a:f*}.
To save computational effort, we seek to combine terms with $\nabla f_*$ at the same step.
To achieve this, first let $e_{-k}$ be $\tilde e_{-k}$ in the second summation,
where $\tilde e = e/\|e\|$;
then, by the invariance of the SRB measure, change the time steps of the terms in \cref{t:expansion} to get
\begin{equation} \begin{split} \label{e:deng}
  U.C.^W
  &= \rho \left[ 
    \sum_{k=0}^\infty \left(
    \frac {\ip{\nabla_{f_*^{k+1} \Psi^s_{-k}}f_*^{k+1}\eta_* e_{-k},f_*^{k+1} e_{-k}}}
    {\|f_*^{k+1} e_{-k}\|^2}
  -\frac{\ip{\nabla_{f_*^{k} \Psi^s_{-k}}f_*^{k}\eta_* e_{-k},f_*^{k} e_{-k}}}
    {\|f_*^k e_{-k}\|^2} \right)
    \right.\\ &\left. 
  + \frac{\ip{\psi_{1} \nabla_{e_{1}}X_{1}, e_{1}}}{\|e_{1}\|^2} 
  - \sum_{k=1}^\infty \left(
    \frac{\ip{\nabla_{f_*^{1-k} \Psi^u_k}f_* \tilde e, f_* \tilde e}}
    {\|f_* \tilde e \|^2}
  - \ip{\nabla_{f_*^{-k} \Psi^u_k} \tilde e, \tilde e} \right)
  \right],
\end{split} \end{equation}
where the subscript $(\cdot)_1$ labels the steps, $\Psi_k:=\psi_k X_k$,
$e$ is the $u$-dimensional hyper-cube defined in \cref{e:e},
$\tilde e$ is the normalized hyper-cube, 
$\eta$ is the projection map along $X=\delta f\circ f^{-1}$ (see \cref{s:define varpi}),
and the definition of $\nabla_{(\cdot)_k}(\cdot)_k$ is in equation~\eqref{e:nablaXY}.

Because $\wedge^uV^u$ is one-dimensional,
$ f_*^k e_{-k} / \|f_*^k e_{-k}\| = \pm \tilde e$ for all $k$.
The sign here is determined by the orientations of $e_k$ and $\tilde e$.
In our algorithm, 
the SRB measure is given by the empirical measure of an aperiodic orbit,
and $ \{ e_k \}_{k\in\Z}$ is obtained by pushing forward on this orbit.
For this case, the orientations of $e_k$'s are the same;
hence, in our paper, we assume
\[ \begin{split}
  \frac{f_*^k e_{n}}  {\|f_*^k e_{n}\|} = \tilde e_{n+k},
  \quad \textnormal{ for any } n, k.
\end{split} \]
Should we consider the more general case,
we only need to keep track of the sign generated by the different orientations of $e_k$'s.
Because this sign is canceled within the inner products in \cref{e:deng},
our work below would not change.

Hence, we may pull $\tilde e$ out of the second term of each summand in  \cref{e:deng}, and define a vector field
\begin{equation} \begin{split} \label{e:p}
p :=&
  \sum_{k=0}^\infty \frac 1 {\|f_*^k e_{-k}\|} 
    P^\perp \nabla_{f_*^{k} \Psi^s_{-k}} f_*^{k}\eta_* e_{-k}
  - \sum_{k=1}^\infty 
    \nabla_{f_*^{-k} \Psi^u_{k}} \tilde e 
  = \sum_{k\ge0} p_{(k)}^\perp - \sum_{k\ge1} p_{(k)}' \,,\\
&\textnormal{ where } \quad
p_{(k)} := \frac{\nabla_{f_*^{k} \Psi^s_{-k}}f_*^{k}\eta_* e_{-k}} 
  { \|f_*^{k} e_{-k}\|} \,,\quad
p_{(k)}^\perp:=P^\perp p_{(k)} \,,\quad
p_{(k)}' := \nabla_{f_*^{-k} \Psi^u_k} \tilde e \,.\;
\end{split} \end{equation}
Here $P^\perp$ is the orthogonal projection operator (appendix \ref{a:projection} definition~\ref{d:projection}), and notice that $p_{(k)}\neq p_k$.
We will see that $P^\perp$ keeps only a convergent part of the first summation, while the normalized $\tilde e$ places the second summation in $\cD^{u\perp}:=P^\perp \cD^u$, which is the orthogonal projection of the space of derivative-like $u$-vectors (appendix~\ref{a:derivative}).
Also note that $p$ is the same for any $e_{-k}$ with the same orientation,
since for any $C^\infty$ function $g>0$,
$P^\perp \nabla ge = g P^\perp\nabla e$.

\begin{lemma} \label{l:p converge}
  $p$ is a convergent summation in $\cD^{u\perp}$.
\end{lemma}

\begin{proof}
Since $\tilde e$ has constant volume, we have $\ip{\nabla \tilde e, \tilde e} =0$.
By the linearity of the projection operator, $p \in \cD^{u \perp}$ should it converge.
To see the convergence of the first summation, use the estimation in equation~\eqref{e:mechanical} and that $P^\perp=P^\perp P^s$ from lemma~\ref{l:compose projection}.
The convergence of the second summation is because $ f_*^{-k} \Psi^u_{k} $ decays exponentially.
\end{proof}

To simplify our writing, we define two maps and show some of their properties.

\begin{definition} \label{d:betaU}
For any $r\in\cD^u, Y\in T_xM$, define
\[ \begin{split} 
  \beta_Y(r)
    &:= \left( f_*r + (\nabla_{Y} f_*)\tilde e \right) / {\|f_* \tilde e\|} . \\
  \tilde \beta(r)
    &:=  \beta_{\tilde v} r + \psi_{1} \nabla_{\tilde e_{1}} X_{1} ,\\
\end{split} \]
\[ \begin{split} 
  U_Y(r)
    &:= \ip{\beta_{Y}(r) , \tilde e_1} - \ip{r, \tilde e},\\
  \tilde U(r)
    &:= \ip{\tilde \beta(r) , \tilde e_1} - \ip{r, \tilde e}
    = U_{\tilde v}(r) + \ip{ \psi_{1} \nabla_{\tilde e_{1}} X_{1} , \tilde e_1 }.
\end{split} \]
Here $\nabla f_*$ is the Riemannian connection of $f_*$ 
(appendix~\ref{a:f*} definition~\ref{d:dnabla}),
$\tilde v$ is the shadowing vector of $\Psi$, 
\[ \begin{split}
  \tilde v := \sum_{k=0}^\infty f_*^{k} \Psi^s_{-k}
  - \sum_{k=1}^\infty f_*^{-k}\Psi^u_k .
\end{split} \]  
\end{definition}

\begin{remark*}
  (1)
  Similar to pushforward operators, for $r\in \cD^u(x)$, we have $\beta_Y (r), \tilde \beta (r) \in \cD^u(fx)$.
  (2) 
  The oblique projections go into the modified shadowing vector $\tilde v$.
  By section~\ref{s:shadowing contribution},
  $\tilde v$ can be efficiently computed using nonintrusive shadowing algorithms,
  which do not require computing oblique projections.
  (3)
  The second-order tangent equation, $\tilde \beta$, 
  is in fact also an inhomogeneous first-order tangent equation with a second-order inhomogeneous term, $\nabla_{\tilde v} f_*$.
\end{remark*}

\begin{lemma} [properties of $\beta$ and $U$] \label{l:beta and U}
For any $r, r'\in\cD^u$, and $X, Y\in T_x\cM$,
\begin{enumerate} [leftmargin=20pt]
  \item $\beta_Xr' \pm  \beta_Yr = \beta_{X\pm Y}(r'\pm r)$,
  $U_Xr' \pm  U_Yr = U_{X\pm Y}(r'\pm r)$;
  \item for both $\beta_Y$ and $\tilde \beta$, we have $\beta r\in\cD^u$,
        $P^\perp \beta r = P^\perp \beta r^\perp$,
        where $r^\perp := P^\perp r$;
  \item for both $U_Y$ and $\tilde U$, $U(r) = U(r^\perp)$.
\end{enumerate}
\end{lemma}

\begin{proof}
(1) By definition. 
(2) By definition, then use lemma~\ref{l:P with f} in appendix~\ref{a:projection}.
(3) Since $\wedge^uV^u$ is one-dimensional, all of its $u$-vectors increase by the same amounts by pushforward.
Hence, $\ip{f_*r^\parallel, \tilde e_1} / \|f_* \tilde e\|= 
      \ip{r^\parallel, \tilde e} $, and
  \[ \begin{split}
    U(r) - U(r^\perp)
    = \ip{\beta(r)-\beta (r^\perp),\tilde e_1}-
    \ip{r -r^\perp, \tilde e}\\
    = \ip{\beta_0(r^\parallel), \tilde e_1}-
    \ip{r^\parallel, \tilde e}
    = \ip{\frac {f_* r^\parallel}{\|f_*\tilde e\|}, \tilde e_1}-
    \ip{r^\parallel, \tilde e}
    =0.
  \end{split} \]
  Here $\beta_0$ means $Y=0$ in definition~\ref{d:betaU}.
\end{proof}

\subsection{Unstable contribution expressed by \texorpdfstring{$p$}{p} and \texorpdfstring{$\tilde U$}{U}}
\hfill\vspace{0.1in}
\label{s:cai}

\begin{lemma} \label{l:UC by p}
  $ U.C.^W = \rho(\tilde U (p))$.
\end{lemma}

\begin{proof}
By the Leibniz rule in appendix~\ref{a:f*} lemma~\ref{l:Leibniz},
$\eta_*=Id$ at $q=0$,
and the definition of $p_{(k)}$ and $p_{(k)}'$ in equation~\eqref{e:p}, we have
\begin{equation} \begin{split} \label{e:dangd}
  \frac {\nabla_{f_*^{k+1} \Psi^s_{-k}} f_*^{k+1} \eta_* e_{-k}}
    {\|f_*^{k+1} e_{-k}\|}
  = \frac 1 {\|f_* \tilde e\|} \left(
    f_*p_{(k)} + (\nabla_{f_*^{k} \Psi^s_{-k}} f_*) \tilde e
    \right)
  = \beta_{f_*^{k} \Psi^s_{-k}}(p_{(k)})
  , \\
  \frac{\nabla_{f_*^{1-k}\Psi^u_{k}} f_* \tilde e} 
    {\|f_* \tilde e\|} 
  = \frac 1 {\|f_* \tilde e\|} 
    \left( f_{*} p_{(k)}'
    + (\nabla_{f_*^{-k} \Psi^u_{k}} f_{*}) \tilde e \right)
  = \beta_{f_*^{-k} \Psi^u_{k}}(p_{(k)}').
\end{split} \end{equation}
The summand in the first summation of equation~\eqref{e:deng} becomes
\[ \begin{split}
  &\frac {\ip{\nabla_{f_*^{k+1} \Psi^s_{-k}}f_*^{k+1}\eta_* e_{-k},f_*^{k+1} e_{-k}}}
    {\|f_*^{k+1} e_{-k}\|^2}
  -\frac{\ip{\nabla_{f_*^{k} \Psi^s_{-k}}f_*^{k}\eta_* e_{-k},f_*^{k} e_{-k}}}
    {\|f_*^k e_{-k}\|^2}
    \\
  =&  \ip{ \beta_{f_*^{k} \Psi^s_{-k}} p_{(k)}, \tilde e_1} 
    - \ip{p_{(k)}, \tilde e}
  = U_{f_*^{k} \Psi^s_{-k}}(p_{(k)})
  = U_{f_*^{k} \Psi^s_{-k}}(p_{(k)}^\perp).
\end{split} \]
The summand in the second summation becomes
\[ \begin{split}
  &\frac{\ip{\nabla_{f_*^{1-k} \Psi^u_k}f_* \tilde e, f_* \tilde e}}
    {\|f_* \tilde e \|^2}
  - \ip{\nabla_{f_*^{-k} \Psi^u_k} \tilde e, \tilde e}
  =  \ip{ \beta_{f_*^{-k} \Psi^u_k} p_{(k)}', \tilde e_1} 
    - \ip{p_{(k)}', \tilde e}
  = U_{f_*^{-k} \Psi^u_k}(p_{(k)}').
\end{split} \]
The lemma is proved by summing over $k$ and lemma~\ref{l:beta and U} (1).
\end{proof}

\subsection{Characterizing \texorpdfstring{$p$}{p} by induction}
\label{s:fastChar}
\hfill\vspace{0.1in}

This subsection shows that $p$ satisfies an inductive relation given by the renormalized second-order tangent equation, 
whose stability indicates that any sequences satisfying this equation on an orbit will eventually converge to $p$.
With the result of the last subsection, we no longer need to compute the propagation of individual terms in the definition of $p$.
This requires only $u$ many first- and second-order tangent solutions.

\begin{lemma} [dynamics of $p$] \label{l:induction}
  $p_{1} := p\circ f = P^\perp \tilde \beta p$. 
\end{lemma}

\begin{remark*}
(1) Should we know the correct $p(x_0)$, 
we can solve all $p_n$ inductively by $p_{n+1} = P^\perp\tilde \beta(p_n)$;
this is more efficient than computing from the definition of $p$,
because the most expensive operation, $\tilde \beta$,
now only needs to operate once on $p$,
instead of many times on each summand in the definition of $p$.
(2) Compared to oblique projections, computing the orthogonal projection is easier and faster;
in particular, it can be done with only $u$ many first order tangent solutions.
\end{remark*}

\begin{proof}
By the definition of $p_1$,
\[ \begin{split}
  p_1
  =& \sum_{k=0}^\infty \frac 
    {P^\perp \nabla_{f_*^{k} \Psi^s_{1-k}}f_*^{k}\eta_* e_{1-k}}  
    {\|f_*^k e_{1-k}\|} 
  - \sum_{k=1}^\infty 
    \nabla_{f_*^{-k}\Psi^u_{1+k}} \tilde e_{1} \\
  =& \sum_{k=-1}^\infty \frac
    {P^\perp \nabla_{f_*^{k+1} \Psi^s_{-k}} f_*^{k+1} \eta_* e_{-k} }  
    {\|f_*^{k+1} e_{-k}\|} 
  - \sum_{k=2}^\infty
    \nabla_{f_*^{1-k}\Psi^u_{k}} \tilde e_1 \\
\end{split} \]
For the first summation, separate the term $k=-1$ and apply \cref{e:dangd},
\[ \begin{split}
\sum_{k=-1}^\infty 
\frac {P^\perp \nabla_{f_*^{k+1} \Psi^s_{-k}} f_*^{k+1} \eta_* e_{-k} } 
{\|f_*^{k+1} e_{-k}\|} 
= P^\perp \frac {\nabla_{\Psi^s_{1}} \eta_* e_1} 
{\|e_1\|} 
+ \sum_{k=0}^\infty P^\perp \beta_{f_*^{k} \Psi^s_{-k}} (p_{(k)})
\end{split} \]
Then, by \cref{l:beta and U}, we can change $p_{(k)}$ to $p^\perp_{(k)}$ for free.

For the second summation, denote $Y:=f_*^{-k} \Psi^u_{k}$, so $p_{(k)}' = \nabla_{Y} \tilde e$, and
\[ \begin{split}
  \nabla_{f_*Y} \tilde e_1
  = \nabla_{f_*Y} \frac{f_* \tilde e}{\|f_* \tilde e\|}
  = \beta_ {Y} p_{(k)}'
    + f_*Y\left(\frac 1 {\|f_*\tilde e\|}\right) f_*\tilde e.
\end{split} \]
Note that $\nabla\tilde e\in \cD^{u\perp}$, $\nabla\tilde e= P^\perp \nabla\tilde e$; since $f_*\tilde e\in \wedge^u V^u_1$, $P^\perp f_*\tilde e =0$.
Hence,
\[ \begin{split}
  \nabla_{f_*Y} \tilde e_1
  = P^\perp \nabla_{f_*Y} \tilde e_1
  = P^\perp  \beta_ {Y} p_{(k)}'.
\end{split} \]

Substituting new expressions of both summations into the expression for $p_1$, we have
\[ \begin{split}
  p_1
  =& \frac {\psi_{1}} {\|e_{1}\|} 
    P^\perp \nabla_{ X^s_{1}} \eta_* e_{1} 
  + \sum_{k=0}^\infty P^\perp  
    \beta_{f_*^{k} \Psi^s_{-k}} p_{(k)}^\perp
  + \psi_{1} \nabla_{ X^u_{1}} \tilde e_{1}
  - \sum_{k=1}^\infty 
    P^\perp  \beta_ {f_*^{-k} \Psi^u_{k}} p_{(k)}'    \\
  =& \frac {\psi_{1}} {\|e_{1}\|} 
    P^\perp \nabla_{ X^s_{1}} \eta_* e_{1} 
  +  \psi_{1} \nabla_{ X^u_{1}} \tilde e_{1}
  + P^\perp \beta_{\tilde v} p 
\end{split} \]
To add the first two terms, use lemma~\ref{l:Richter} on the first term, 
and that $P^\perp e=0$,
\[ \begin{split}
  \frac {\psi} {\|e\|} P^\perp \nabla_{ X^s} \eta_* e
  &= \frac {\psi} {\|e\|} P^\perp \left( \nabla_{e} X - \nabla_{X^u} e \right)
  = \psi P^\perp \left( \nabla_{\tilde e} X - \frac 1 {\|e\|} \nabla_{X^u} e 
    - e X^u (\frac 1{\|e\|})\right)
  \\
  &= \psi P^\perp (\nabla_{\tilde e} X -  \nabla_{X^u} \tilde e) 
  = \psi P^\perp \nabla_{\tilde e} X -  \psi \nabla_{X^u} \tilde e.
\end{split} \]
Hence $p_1 = \psi_1 P^\perp \nabla_{\tilde e_1} X_1 +  P^\perp \beta_{\tilde v} p$, as claimed.
\end{proof}

Hence, we have found that the vector field $p$ is invariant under this dynamics.
We may also define a sequence $r$ on each orbit only from the dynamics of $p$.
It turns out that $r$ convergences exponentially fast to $p$ on any orbit.

\begin{lemma} [stability of renormalized second-order tangent equation] \label{l:forget initial}
  For any $x_0\in K$, $r_0\in\cD^u(x_0)$, 
  define a sequence $\{r_n\}_{n\ge0}$ by the dynamics of $p$,
  \[ \begin{split}
  r_{n}\in \cD^u(x_n), \quad
  r_{n+1}:= P^\perp \tilde \beta r_n .
  \end{split} \]
  Then $\{r_n\}_{n\ge0}$ approximates $p$, that is,
  \[ \begin{split}
    \lim_{n\rightarrow \infty} r_n - p_n(x_0) =
    \lim_{n\rightarrow \infty} P^\perp \tilde\beta^n r_0 - P^\perp \tilde\beta^n p(x_0) = 0.
  \end{split} \]
\end{lemma}

\begin{remark*}
  (1) An easy choice of $r_0$ is zero.
  (2) The notation $\beta^n$ means applying the inhomogeneous propagation operator $n$ times.
  This is done similarly to the pushforward operator, with $\tilde v$ and $\tilde e$ evaluated at suitable steps.
\end{remark*}

\begin{proof}
  $ P^\perp \tilde\beta^n r_0 - P^\perp \tilde\beta^n p
  = {P^\perp f_*^n (r_0-p)} / { \|f_*^n \tilde e\|}
  = {P^\perp f_*^n P^s (r_0-p)} / { \|f_*^n \tilde e\|}
  \le C \lambda^{2n}$.
  Where we used \cref{l:compose projection}.
\end{proof}

Finally, we can prove the main theorem of this paper.
Readers may refer to the start of \cref{s:main results} to recall the main steps we followed to reach this point.

\begin{theorem} [fast response formula for the unstable contribution] \label{t:lra}
Make the same assumptions as theorem~\ref{t:ruelle}.
Let $\{x_n:=f^nx_0\}_{n\ge0}$ be an orbit on the attractor,
  for any $r_0\in \cD^u(x_0)$, define the sequence $\{r_n\}_{n\ge0}$ with $r_{n}\in \cD^u(x_n)$ by
  \[ \begin{split}
  r_{n+1}:= P^\perp \tilde \beta r_n .
  \end{split} \]
  Then for almost all $x_0$ according to the SRB measure,
  the unstable contribution,
  \[ \begin{split}
    U.C.^W 
    = \lim_{N \rightarrow \infty} \frac 1 {N} \sum_{n=0}^{N-1}
      \ip{\tilde \beta r_n ,
      \tilde e_{n+1} }.
  \end{split} \]
\end{theorem}

\begin{remark*}
  (1)
  In practice, a $\rho$-typical point $x_0$ can be found by running
  almost all orbits starting from the attractor basin for some time.
  (2)
  In practice, the unstable subspace, $V^u$, can be obtained by pushing-forward $u$ many
  randomly initiated vectors, because unstable vectors grow faster than stable ones.
  (3) 
  A more careful analysis should prove this theorem for almost all $x_0$ in the attractor basin and almost all initial guesses of $V^u(x_0)$.
\end{remark*} 

\begin{proof} 
By the same tail bound as in lemma~\ref{l:arau} and~\ref{l:hao}, 
we can show that $\ip{p,\tilde e}$ is continuous on the attractor $K$.
Hence, $\tilde U(p)$ is continuous on $K$;
by the ergodic theorem and lemma~\ref{l:UC by p},
we see that almost surely according to the SRB measure,
\[ \begin{split}
  &U.C.^W 
  = \lim_{N \rightarrow \infty} \frac 1{N} \sum_{n=0}^{N-1} \tilde U (p_n).
\end{split} \]
By lemma~\ref{l:induction},
the definition of $\tilde U$, the orthogonal condition, we have
\[ \begin{split}
  &U.C.^W 
  = \lim_{N \rightarrow \infty} \frac 1{N} \sum_{n=0}^{N-1} \tilde U ((P^\perp\tilde \beta)^n p) \\
  =& \lim_{N \rightarrow \infty} \frac 1{N} \sum_{n=0}^{N-1} 
    \ip{\tilde \beta (P^\perp\tilde \beta)^n p, \tilde e_{n+1}}
    - \ip{(P^\perp\tilde \beta)^n p, \tilde e_{n}}
  = \lim_{N \rightarrow \infty} \frac 1{N} \sum_{n=0}^{N-1} 
    \ip{\tilde \beta (P^\perp\tilde \beta)^n p, \tilde e_{n+1}}
\end{split} \]
Finally, apply lemma~\ref{l:forget initial} to replace $p$ by its orbit-wise approximation $r$.
\end{proof}

\section{Fast response algorithm}
\label{s:algorithm}

This section concerns the algorithm for computing the new fast response formula.
First, to further reduce computational cost and round-off error, we show that the renormalization only needs to be done intermittently.
We will also write major equations in matrix notation, which is more suitable for coding.
Then we give a pseudocode of the algorithm.

\subsection{Intermittent renormalization}
\hfill\vspace{0.1in}

This subsection explains how renormalization only needs to be done once after a segment of several steps.
Here, renormalization refers to
\begin{itemize}
  \item The orthogonal projection $P^\perp$ in the fast formula, defined in appendix~\ref{a:projection}.
  \item The rescaling in $\beta$, that is, 
    dividing by the $\|f_*{\tilde e}\|$ factor in definition~\ref{d:betaU}.
  \item Othonormalizing the basis vectors of $V^u$.
\end{itemize}
The third operation does not appear explicitly in the fast formula, but a good basis improves numerical performance.
We show one by one that all three renormalizations only need to be done every once a while, and that they can be done together using a short formula in matrix notation.

The subscript convention for multiple segments, shown in figure~\ref{f:subscript}, 
is similar to that of the nonintrusive shadowing algorithm \cite{Ni_NILSS_JCP,Ni_nilsas}.
We divide an orbit into small segments, each containing $N$ steps.
The $\alpha$-th segment consists of steps $\alpha N$ to $\alpha N+N$,
where $\alpha$ runs from $0$ to $A-1$;
notice that the last step of segment $\alpha$ is also the first step of segment $\alpha+1$.
We use double subscript, such as $x_{\alpha, n}$,
to indicate the $n$-th step in the $\alpha$-th segment,
which is the $(\alpha N+n)$-th step in total.
Note that some quantities have two values on the step between two segments since renormalization is performed: for example, $e_{\alpha, N}\neq e_{\alpha+1, 0}$.
Continuity across interfaces is true only for some quantities, 
such as shadowing vectors $v$, $\tilde v$, and unit unstable cube $\tilde e$.
Later, we will define some quantities on the $\alpha$-th segment,
such as $C_\alpha, d_\alpha$ in equation~\eqref{e:C},
their subscripts are the same as the segment they are defined on.
For quantities to be defined at the interfaces, such as $Q_\alpha, R_\alpha, b_\alpha$ in \cref{e:QR}, their subscripts are the same as the total number of steps at the interface, divided by $N$.

\begin{figure}[ht] \centering
  \includegraphics[width=0.7\textwidth]{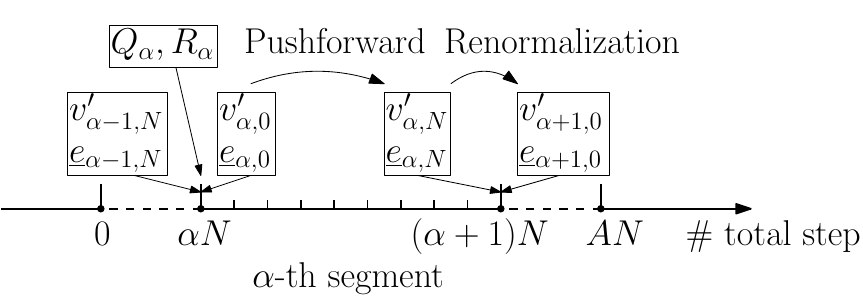}
  \caption{Subscript convention for multiple segments.}
  \label{f:subscript}
\end{figure}

\begin{lemma} [intermittent orthogonal projection]  \label{l:violin}
For any $\tilde r_{0,0} \in \cD^u$,
let 
\[ \begin{split}
\tilde r_{\alpha, n}:=\tilde \beta\tilde r_{\alpha, n-1}, \quad 
\tilde r_{\alpha+1, 0}:= P^\perp \tilde r_{\alpha, N},
\end{split} \]
then almost surely 
\[ \begin{split}
  U.C.^W 
  = \lim_{A \rightarrow \infty} \frac 1 {NA} \sum_{\alpha=0}^{A-1}
  \ip{ \tilde  r_{\alpha, N} , \tilde e_{\alpha, N} }.
\end{split} \]
\end{lemma}

\begin{proof}
Denote $(P^\perp\tilde \beta)^{\alpha N}\tilde r_{0,0}$ by $r'$,
by lemma~\ref{l:beta and U}, lemma~\ref{l:UC by p},
and the definition of the SRB measure,
the unstable contribution from step $\alpha N$ to $\alpha N + N-1$ is,
\[ \begin{split}
  \sum_{n=0}^{N-1} \tilde U ((P^\perp\tilde \beta)^{\alpha N+n} r_{0,0})
  = \sum_{n=0}^{N-1} \tilde U ((P^\perp\tilde \beta)^{n} r')
  = \sum_{n=0}^{N-1} \tilde U (P^\perp\tilde \beta^n  r')
  = \sum_{n=0}^{N-1} \tilde U (\tilde \beta^n  r') \\
  = \sum_{n=0}^{N-1} \ip{\tilde \beta^{n+1} r', \tilde e_{\alpha N+n+1}} 
  - \ip{\tilde \beta^n r', \tilde e_{\alpha N +n}} 
  =  \ip{\tilde \beta^N r', \tilde e_{\alpha N + N}} .
\end{split} \]
Average over all segments and adopt the subscript convention to prove the lemma.
\end{proof}

\begin{lemma} [intermittent rescaling] \label{l:inter rescale}
Let $e$ be the first-order tangent solution,
\[ \begin{split}
  e_{0,0} = \tilde e_{0,0}, \;
  e_{\alpha, n}:=f_* e_{\alpha, n-1},\;
  e_{\alpha+1, 0}:= e_{\alpha, N} / \|e_{\alpha, N}\| = \tilde e_{\alpha+1, 0};
\end{split} \]
For any $r_{0,0} \in \cD^u$, let $r$ be governed by the second-order tangent equation,
\[ \begin{split}
  r_{\alpha, n}:= f_* r_{\alpha, n-1} + (\nabla_{\tilde v_{\alpha, n-1}} f_*) e_{\alpha, n-1}
    + \psi_{\alpha, n} \nabla_{ e_{\alpha, n}} X_{\alpha, n} ,\;\;
  r_{\alpha+1, 0}:= P^\perp r_{\alpha, N} / \|e_{\alpha, N}\| .
\end{split} \]
Then almost surely 
\[ \begin{split}
  U.C.^W 
  = \lim_{A \rightarrow \infty} \frac 1 {NA} \sum_{\alpha=0}^{A-1}
  \frac {\ip{   r_{\alpha, N} ,  e_{\alpha, N} }} {\ip{e_{\alpha, N} ,  e_{\alpha, N} }}.
\end{split} \]
\end{lemma}

\begin{proof}
We prove by induction that if we choose $r_{0,0} = \tilde r_{0,0}$,
then $r_{\alpha, n} = \|e_{\alpha,n}\| \tilde r_{\alpha, n}$.
Within segment $\alpha$, assuming that we have this relation for $n-1$, then
\[ \begin{split}
r_{\alpha,n} = \|e_{\alpha,n-1}\| f_* \tilde r_{\alpha, n-1}
  +\|e_{\alpha,n-1}\| (\nabla_{\tilde v_{\alpha, n-1}} f_*) \tilde e_{\alpha, n-1}
  + \|e_{\alpha,n}\| \psi_{\alpha, n} \nabla_{\tilde e_{\alpha, n}} X_{\alpha, n} 
\\
=  \|e_{\alpha,n}\| \tilde \beta  \tilde r_{\alpha, n-1}
=  \|e_{\alpha,n}\|  \tilde r_{\alpha, n}.
\end{split} \]
Hence, the relation also holds for $n$;
it also holds across interfaces, since
\[ \begin{split}
  r_{\alpha+1, 0} 
  := P^\perp r_{\alpha, N} / \|e_{\alpha, N}\|
  = P^\perp \tilde r_{\alpha, N}
  =: \tilde r_{\alpha+1, 0}
  = \tilde r_{\alpha+1, 0} \|e_{\alpha+1, 0}\|,
\end{split} \]
where $\|e_{\alpha+1, 0}\|=1$ by construction.
Finally, substitute into lemma~\ref{l:violin}.
\end{proof}

\begin{proposition} [intermittent renormalization] \label{l:titan}
Neglecting the first two subscripts $\alpha$ and $n$, let $e:=\wedge_{i=1}^u e_i$, $r:=\sum_i e_1\wedge\cdots\wedge r_i \wedge\cdots\wedge e_u$.
Denote matrices $\underline e:=[e_1,\cdots,e_u]$, $\underline r:=[r_1,\cdots,r_u]$.
Then the renormalization in lemma~\ref{l:inter rescale} is realized by the following matrix operations:
\[ \begin{split}
  \underline e_{\alpha, N} &= Q_{\alpha+1} R_{\alpha+1}, \quad
  \underline e_{\alpha+1, 0} = Q_{\alpha +1},\\
  \underline r^\perp_{\alpha, N} 
    &= \underline r_{\alpha, N} - Q_{\alpha+1} Q_{\alpha+1}^T \underline r_{\alpha, N} ,\quad
  \underline r_{\alpha+1, 0} = \underline r^\perp_{\alpha, N} R^{-1}_{\alpha+1}.
\end{split} \]
Here, the first equation means performing the QR factorization,
and $Q^T\underline r:= [\ip{Q_i, r_j}] $ is a matrix.
Using $\Tr(\cdot)$ to denote the trace of a matrix, the unstable contribution is
\[ \begin{split}
  U.C.^W 
  =\lim_{A \rightarrow \infty}  \frac 1 {NA} \sum_{\alpha=0}^{A-1}
    \Tr\left( R^{-1}_{\alpha+1} Q_{\alpha+1} ^T \underline r_{\alpha, N}  \right).
\end{split} \]
\end{proposition}

\begin{remark*}
(1) 
Rewriting $r$ on the new basis does not change $r$ as a $u$-vector, but improves the numerical performance.
(2)
Computing $\underline e$ via pushforwards and renormalization is also part of the nonintrusive showing algorithm.
(3)
The pushforward relation inside a segment is the same as that in lemma~\ref{l:inter rescale}.
\end{remark*}

\begin{proof}
The renormalization on $e$ is due to the definition of QR factorization.
For $r$, first substitute the QR factorization in the expression for projection, to show
\[ \begin{split}
  \underline r_{\alpha, N}^\perp
  = \underline r_{\alpha, N} - \underline e_{\alpha, N}
    (\underline e_{\alpha, N} ^T \underline e_{\alpha, N})^{-1} 
    (\underline e_{\alpha, N} ^T \underline r_{\alpha, N} )
  = \underline r_{\alpha, N} - Q_{\alpha+1} Q_{\alpha+1}^T \underline r_{\alpha, N}.
\end{split} \]
To write the rescaling in lemma~\ref{l:inter rescale}, $r_{\alpha+1, 0} := r^\perp_{\alpha, N} / \|e_{\alpha, N}\|$,
on the new basis $Q_{\alpha +1}$, use \cref{l:change of basis} to show
\[ \begin{split}
  \underline r_{\alpha+1, 0}
  = \det(R_{\alpha+1})\, \underline r^\perp_{\alpha, N} R^{-1}_{\alpha+1} / \|e_{\alpha,N}\|
  = \underline r^\perp_{\alpha, N} R^{-1}_{\alpha+1},
\end{split} \]  
since $\det(R_{\alpha+1}) = \|e_{\alpha,N}\|$ by definition of QR factorization.
Hence, we have proved that the renormalization for $r$ in lemma~\ref{l:inter rescale} is realized by the matrix operations in the lemma.

To simplify the expression for $U.C.^W$, 
first use lemma~\ref{l:baozi} in appendix~\ref{a:projection},
\[ \begin{split}
  \ip{ r_{\alpha, N} , e_{\alpha, N} }
  = \ip{ r^\parallel_{\alpha, N} , e_{\alpha, N} }.
\end{split} \]
Because $r_i^\parallel\in V^u$,
we can write it as 
\[ \begin{split}
r_i^\parallel = (e^j_u r_i) e_j,  
\end{split} \]
where $e^j_u$ is the covector of $e_j$ in $\cV^u$,
and $i,j$ are summed from $1$ to $u$.
Hence, 
\[ \begin{split}
r^\parallel =\sum_{i=1}^u e_1\wedge\cdots\wedge (e^j_u r^\parallel_i) e_j \wedge\cdots\wedge e_u  
= \sum_{i=1}^u e_1\wedge\cdots\wedge (e^i_u r^\parallel_i) e_i \wedge\cdots\wedge e_u,
\end{split} \]
because the wedge product between parallel vectors is zero.
Moreover, by lemma~\ref{l:inter rescale},
\[ \begin{split}
  &U.C.^W 
  = \lim_{A\rightarrow\infty} \frac 1 {NA} \sum_{\alpha=0}^{A-1}
    \frac {\ip{ r_{\alpha, N} , e_{\alpha, N} }} {\ip{ e_{\alpha, N} , e_{\alpha, N}}}
  = \lim_{A\rightarrow\infty} \frac 1 {NA} \sum_{\alpha=0}^{A-1} \sum_{i=1}^u
    e^i_{u,\alpha, N} r_{\alpha, N,i}^\parallel .
\end{split} \]

To write this expression in matrix notation,
first notice that $(e^i_u r^\parallel_l) e_i = r^\parallel_l$, so
\[ \begin{split}
(e^i_u r^\parallel_l) \ip{e_i,e_j} 
= \ip{ r^\parallel_l, e_j}
= \ip{ r_l, e_j} .
\end{split} \]
This is a linear equation system with $u$ many equations whose solution is
\[ \begin{split}
  e^i_u r^\parallel_l = i\textnormal{-th entry of the vector } 
  (\underline e^T \underline e)^{-1} (\underline e^T r_l).
\end{split} \]
Hence we can further write the expression of $U.C.^W$ in matrix notation,
\[ \begin{split}
  U.C.^W 
  =\lim_{A\rightarrow\infty}  \frac 1 {NA} \sum_{\alpha=0}^{A-1}
    \Tr\left( (\underline e_{\alpha, N} ^T \underline e_{\alpha, N})^{-1} 
    (\underline e_{\alpha, N} ^T \underline r_{\alpha, N} ) \right).
\end{split} \]
Substituting the QR factorization of $\underline e$, we have
\[ \begin{split}
  (\underline e^T \underline e)^{-1} (\underline e^T \underline r)
  = (R^T Q^T Q R)^{-1} R^T Q^T \underline r
  = R^{-1}Q^T\underline r.  
\end{split} \]
\end{proof}

\subsection{Pseudocode} \label{s:procedure list}
\hfill\vspace{0.1in}

This subsection gives a pseudocode of the fast response algorithm.
When $\cM=\R^M$, the corresponding simplifications are explained.
All sequences in this subsection are defined by some inductive relations on an orbit.
In particular, here the unstable vectors $\underline e$ and shadowing vectors $v, \tilde v$ are orbit-wise approximations of the true values.
We use $r$ as the orbit-wise version of $p$, whose approximation was established by this paper.
The explanation of the subscripts is given in figure~\ref{f:subscript}.

\begin{enumerate}[label={\roman*.}]

\item 
Evolve the dynamical system for a sufficient number of steps before $n=0$,
so that $x_0$ is on the attractor at the beginning of our algorithm. 
Then, evolve the system from segment $\alpha=0$ to $\alpha=A-1$,
each containing $N$ steps, to obtain the orbit,
\[ \begin{split}
  x_{\alpha, n+1} = f(x_{\alpha, n}), \quad x_{\alpha+1,0} = x_{\alpha,N}.
\end{split} \]

\item 
Start with initial condition $v'=0$ and $\tilde v'=0$, and random initial conditions 
for each column in $\underline e:=[e_1,\cdots,e_u]$.
Then, repeat the following procedures for all $\alpha$.
\begin{enumerate}
  \item
  From initial conditions, solve first-order tangent equations ($\alpha$ neglected),
  \[ \begin{split}
    \underline e_{n+1} = f_* \underline e_{ n},\quad
    v'_{n+1} = f_* v'_{n} + X_{ n+1}, \quad
    \tilde v'_{ n+1} = f_* \tilde v'_{n} + \Psi_{n+1}.
  \end{split} \]
  Here $X:=\delta f\circ f^{-1}=(\partial f/\partial \gamma) \circ f^{-1}$, 
  where $\gamma$ is the parameter of the dynamical system;
  $\Psi_{n+1} := \psi_{n+1} X_{n+1}$, 
  $\psi := \sum_{m=-W}^W (\Phi \circ f^m-\rho(\Phi))$ for a large $W$,
  $f_*$ is the pushforward operator.
  When $\cM=\R^M$, 
  \[ \begin{split}
    X_{n+1} = (\partial f/\partial \gamma)_n = [\partial f^i/\partial \gamma]_{n,i}\,,
    \quad \textnormal{} \quad 
    f_* = [\partial f^i/\partial z^j]_{ij},
  \end{split} \]
  where $[\cdot]_{ij}$ is the matrix with $(i,j)$-th entry given inside the bracket,
  $f^i$ is the $i$-th component of $f$, $z^j$ is the $j$-th coordinate of $\R^M$.

  \item
  Compute and store the covariance matrix and the inner product,
  \begin{equation} \begin{split} \label{e:C}
    C_\alpha &:= \sum_{n=0}^{N} {}' \underline e_{\alpha,n}^T \underline e_{\alpha,n}
    := \frac 12 \underline e_{\alpha,0}^T \underline e_{\alpha,0}
    + \sum_{n=1}^{N-1}\underline e_{\alpha,n}^T \underline e_{\alpha,n}
    + \frac 12 \underline e_{\alpha,N}^T \underline e_{\alpha,N},\\ 
    d_\alpha &:= \sum_{n=0}^{N} {}' \underline e_{\alpha,n}^T v'_{\alpha,n}, \quad
    \tilde d_\alpha := \sum_{n=0}^{N}{}' \underline e_{\alpha,n}^T \tilde v'_{\alpha,n},
  \end{split} \end{equation}
  where $\sum'$ is the summation with $1/2$ weight at the two end points.
  Here $\underline e^T \underline e := [\ip{e_i, e_j}]$ is a matrix,
  same for $\underline r^T \underline e$ and $Q^T\underline r$ later.
  
  \item 
  At step $N$ of segment $\alpha$, 
  orthonormalize $\underline e$ with a QR factorization, and compute
  \begin{equation} \label{e:QR}
    \underline e_{\alpha, N} = Q_{\alpha+1} R_{\alpha+1}, \quad
    b_{\alpha+1} = Q_{\alpha+1}^T v'_{\alpha, N}, \quad
    \tilde b_{\alpha+1} = Q_{\alpha+1}^T \tilde v'_{\alpha, N}.
  \end{equation}
  
  \item
  Set initial conditions of the next segment, 
  \[
    \underline e_{\alpha+1, 0} = Q_{\alpha +1},\quad
    v'_{\alpha+1, 0} = v'_{\alpha, N} - Q_{\alpha+1} b_{\alpha+1} ,\quad
    \tilde v'_{\alpha+1, 0} = \tilde v'_{\alpha, N} - Q_{\alpha+1} \tilde b_{\alpha+1} .
  \]
\end{enumerate}

\item 
Solve the nonintrusive shadowing problem,
\[\begin{split}
  &\min_{\{a_\alpha\}} \sum_{\alpha=0}^{A-1} 2 d_\alpha^T a_\alpha
  + a_\alpha^T C_\alpha a_\alpha \\
  \mbox{s.t. }& 
  a_{\alpha} = R_{\alpha} a_{\alpha-1} + b_{\alpha},
  \quad \alpha=1,\ldots,A-1.
\end{split}\]
We solve this via the Schur complement, which is a tridiagonal block matrix, and it can be solved with $O(A)$ cost:
similar ideas were first seen in \cite{Blonigan_2017_adjoint_NILSS}, then adapted to our case in section 3.4 of \cite{Ni_fdNILSS}.
Solve the same problem again, with $b$ replaced by $\tilde b$, for $\tilde \alpha$.
Then compute $ v_\alpha $ and $\tilde v_\alpha$,
\[
 v_\alpha = v'_\alpha + \underline e_\alpha a_\alpha, \quad
 \tilde v_\alpha = \tilde v'_\alpha + \underline e_\alpha \tilde a_\alpha.
\]

\item
Compute the shadowing contribution,
\[ 
  S.C.
  = \lim_{A\rightarrow\infty}
    \frac 1 {AN} \sum_{\alpha=0}^{A-1} \sum_{n=0}^N{}'
     v_{\alpha,n}( \Phi_{\alpha,n}),
\]
where $v(\cdot)$ means to differentiate a function in the direction of $v$.

\item 
Denote $\underline r:=[r_1,\cdots,r_u]$.
Set initial condition $r_{0,i} =0$ for all $1\le i\le u$.
Then repeat the following procedures for all $\alpha$.
\begin{enumerate}
  \item  
  From initial conditions, solve second-order tangent equations, $\alpha$ neglected,
  \[ \begin{split}
  r_{n+1,i} 
  = f_*r_{n,i} 
    + (\nabla_{\tilde v_n} f_*) e_{n,i} 
    + \psi_{n+1} \nabla_{e_{n+1,i}} X_{n+1} 
    \;.
  \end{split} \]
  Here $\nabla_{(\cdot)}f_*$ is the Riemannian connection of the pushforward operator
  defined in appendix~\ref{a:f*}. 
  In $\R^M$, $\nabla f_*$ is the derivative of Jacobian,
  or the Hessian tensor, and 
  \[ \begin{split}
    \nabla_Y f_* = [Y (\partial f^i /\partial z^j)]_{ij}    .
  \end{split} \]
  To compute $\nabla_{e_{n+1,i}} X_{n+1}$, 
  denote the coordinate at $x_n$ and $x_{n+1}$ by $\zeta$ and $z$, then
  \[ \begin{split}
  &\nabla_{e_{n+1,i}} X_{n+1} 
  = \nabla_{f_* e_{n,i}} X_{n} \circ f
  = \nabla_{f_* e_{n,i}} \delta f_{n} 
  = \nabla_{f_* e_{n,i}} \left(\delta f_{n}^j \pp{}{z^j} \right) \\
  =& f_* e_{n,i} \left(\delta f_{n}^j \right)  \pp{}{z^j} 
  + \delta f_{n}^j \nabla_{f_* e_{n,i}} \pp{}{z^j} 
  = e_{n,i}^l \pp {\delta f_{n}^j} {\zeta^l} \pp{}{z^j} 
  + \delta f_{n}^j \nabla_{f_* e_{n,i}} \pp{}{z^j}. 
  \end{split} \]
  Here $\pp{}{z_j}$ is the $j$-th coordinate vector.
  In $\R^M$, both coordinates $\zeta$ and $z$ are the canonical coordinate, both denoted by $z$,
  and $\nabla \pp{}{z_j}$ is zero, hence
  \[ \begin{split}
  \nabla_{e_{n+1,i}} X_{n+1} 
  = e_{n,i}^l \pp {\delta f_{n}^j} {z^l} \pp{}{z^j} 
  = [e_{n,i} (\delta f_{n}^j)]_j.
  \end{split} \]
  This is a vector at $x_{n+1}$,
  where $[\cdot]_j$ is a vector in $\R^M$ whose $j$-th entry is given in the bracket.
  The differentiation $e_{n,i}(\delta f_{n}^j)$ happens at $x_n$,
  when $\delta f_{n}^j$ is a function in a neighborhood of $x_n$.

  \item
  To set initial conditions of the next segment, first orthogonally project,
  \[
  \underline r_{\alpha, N}^\perp
  = \underline r_{\alpha, N} - Q_{\alpha+1} Q_{\alpha+1}^T \underline r_{\alpha, N}.
  \]
  Then change basis and rescale,
  \[ \begin{split}
    \underline r_{\alpha+1, 0} = \underline r^\perp_{\alpha, N} R^{-1}_{\alpha+1}.
  \end{split} \]
\end{enumerate}

\item 
Let $\Tr(\cdot)$ be the trace of a matrix,
compute the unstable contribution,
\[ \begin{split}
  U.C.^W 
  = \lim_{A\rightarrow\infty}
  \frac 1 {NA} \sum_{\alpha=0}^{A-1}
    \Tr\left( R^{-1}_{\alpha+1} Q_{\alpha+1} ^T \underline r_{\alpha, N}  \right).
\end{split} \]

\item
The linear response is 
\[ \begin{split}
  \delta \rho(\Phi) =\lim_{W\rightarrow \infty} S.C. - U.C.^W.
\end{split} \]

\end{enumerate}

\section{Discussions}
\label{s:discuss}

\subsection{Notes on implementation}
\label{s:remarks}
\hfill\vspace{0.1in}

When the number of homogeneous tangent solutions computed, $u'$, is strictly larger than $u$,
the unstable contribution part of our algorithm may or may not work,
depending on whether the renormalized second order tangent solutions have a meaningful average.
This is different from the nonintrusive shadowing algorithm,
which works for any $u'\ge u$.
It remains to investigate whether and how much error is incurred when using a large $u'$,
especially how the error relates to the spectrum of the Lyapunov exponents.

We discuss how to choose $N$, the number of steps in each segment.
A larger $N$ leads to fewer segments for the same total number of steps, thus saving some computational cost.
However, the major computational cost in the algorithm comes from computing first- and second order tangent equations.
Hence, the benefit of choosing a very large $N$ is limited.
Still, if we really want a large $N$, 
the upper bound is from the numerical stability of first- and second order tangent equations.
Notice that second order tangent equations are essentially first order tangent equations 
with a second-order inhomogeneous term.
Hence, the limiting factors for large $N$ is that, over one segment, the $u$ many first-order homogeneous tangent solutions can not grow too large or too parallel to each other.

There are several places in the algorithm where we contract a high-order tensor
with several one-dimensional vectors,
which can be done more efficiently via the so-called `vectorized' programming.
More specifically, when contracting one tensor with several vectors,
we should load the tensor into the computer once and take the inner-product with all vectors, 
instead of loading the tensor once for each vector \cite{Ni_nilsas,Ni_CLV_cylinder}.
This vectorization can be used in the nonintrusive shadowing algorithm,
where several first-order tangent solutions, $v'$ and $\{e_i\}_{i=1}^u$,
are multiplied with the same Jacobian matrix $f_*$, which is a two-dimensional tensor.
Vectorization can also be used when solving second-order tangent equations,
for the contraction between $f_*$ and $\{r_i\}_{i=1}^u$, 
$\nabla_{\tilde v} f_*$ and $\{e_i\}_{i=1}^u$, and $\nabla X$ and $\{e_i\}_{i=1}^u$.

Finally, we remind readers that our result is for discrete-time systems.
In numerical simulation, continuous-time systems are discretized, and it seems that we can just apply our current algorithm.
However, it is still better to use algorithms based on continuous-time theories, and a separate treatment of the time direction gives better performances.
The continuous-time theory was given in \cite{Ni_asl,vdivF}, and the algorithm for the shadowing contribution was given in \cite{Ni_NILSS_JCP,Ni_adjoint_shadowing}.
Similarly, a slightly different treatment should be needed for delayed-time systems.

\subsection{Cost estimation}
\label{s:cost estimation}
\hfill\vspace{0.1in}

Depending on the sparsity of the problem,
the main computational complexity of the fast response may come from different parts.
In the worst case, the Hessian tensor $\nabla f_*$ is dense, and contracting it with $\tilde v$ takes $O(M^3)$ operations.
If we further assume that each entry in $\nabla f_*$ requires $O(1)$ operations to obtain, then it also takes $O(M^3)$ operations to obtain $\nabla f_*$.
Hence, for the dense case, the cost for each step of the fast response algorithm is dominated by obtaining $\nabla_{\tilde v}f_*$, 
which only needs to be done once, and the cost is typically $O(M^3)$.

On the other hand, for many engineering problems, 
for example, fluid mechanics and image processing,
$\nabla f_*$ and $f_*$ typically have only $O(M)$ entries,
with each entry taking $O(1)$ operations to compute.
The cost then comes mainly from computing terms such as $f_*\underline e$,
whose complexity is $O(uM)$ for each step,
and can be computed faster via vectorized programming.
The sparse case is perhaps more common in real-life applications.

Finally, we give a very crude and formal estimate on 
the error and total cost of fast response,
using the decorrelation step number $W$, the total number of steps $T:=AN$,
the unstable dimension $u$, and the system dimension $M$.
First, the error in using finite $W$, $ U.C.-U.C.^W$, 
is $O (\theta^W)$ for some $0<\theta<1$, 
which is the decorrelation rate.
We make the simplifying assumption that, 
on any orbit, for any mean-zero function $\phi$,
\[ \begin{split}
  \sum_{n=1}^N \phi(x_n) \sim O(\sqrt N).
\end{split} \]
This assumption is verified later in our numerical example,
and it can be proved under hyperbolicity.

Passing $\Phi$ to $\Phi-\rho(\Phi)$ makes $\rho(\Phi)=0$, so $\psi \sim O(\sqrt{W})$;
further notice that $\psi X$ is the inhomogeneous term for $\tilde v$,
and that $\tilde v$ and $\psi$ are the inhomogeneous terms in $\tilde \beta$;
hence, $\tilde U(p) \sim O(\sqrt{W})$.
Hence, the error caused by averaging $\tilde U(p)$ on a finite orbit of length $T$ is
$\sim O(\sqrt{W/T})$.
Hence, the total error for the unstable contribution, 
which is also the major error for the fast response, denoted by $h$, is
\[ \begin{split}
  h\sim O(\theta^W)  + O(\sqrt{W/T})
\end{split} \]
In practice, we want the two errors to be roughly equal to each other,
hence $T\sim {\theta^{-2W}}W$.
By our discussion of the cost at each step, for typical problems with sparsity, the numerical complexity of fast response is
\[ \begin{split}
  \textnormal{Complexity of FR} 
  \sim O(uM \theta^{-2W}W).
\end{split} \]

In comparison, for the path-perturbation formula,
the size of the integrand is $O( \lambda_{max}^W)$,
where $\lambda_{max}>1$ is the largest Lyapunov exponent.
By similar arguments, 
\[ \begin{split}
  h\sim O(\theta^W)  + O(\lambda_{max}^W/\sqrt{N'})
  ,\quad
  T=N'W\sim \theta^{-2W}  \lambda_{max}^{2W} W .
\end{split} \]
Here, $N'$ is the number of sample orbits, each with $W$ steps,
and $T$ is the total number of steps in all sample orbits.
Note that only one first-order tangent equation and one $f$ are computed at each step;
hence, for the sparse case, 
the cost per step is $O(M)$, and in total, 
\[ \begin{split}
  \textnormal{Complexity of the path-perturbation method} 
  \sim O(M \theta^{-2W}\lambda_{max}^{2W} W),
\end{split} \]
Hence, for typical problems with sparsity,
the numerical complexity of the path-perturbation formula
can be about $O(\lambda_{max}^{2W}/u)$ times higher than the fast response.

\subsection{Beyond uniform hyperbolicity}
\label{s:beyond hyper}
\hfill\vspace{0.1in}

For deterministic dynamical systems, there are several levels of non-uniform hyperbolicity and non-hyperbolicity.
We discuss whether the fast response method is affected and, if so, how to potentially fix it.

The first level is that the system is non-uniformly hyperbolic, but we can prove that the linear response exists and that conventional formulas are correct.
Then, the equivalence of other linear response formulas with fast response formulas can be proved more easily than the existence of the linear response.
Hence, our results should be theoretically correct whenever any linear response formula gives the correct sensitivity.
It is mainly because of the simplicity of the discussion that we assumed uniform hyperbolicity in this paper.

In terms of numerics, so far, it seems that the fast response algorithm works whenever we can prove the existence of the linear response, even when the system slightly fails the uniform hyperbolicity assumption and the regularity assumption.
We applied the fast response algorithm on the Pomeau-Manneville map in the appendix of \cite{GN25}, where the map has only summable but not exponential decay of correlations, and the result is good (the proof of linear response was given in \cite{BT16,bahsoun16,Korepanov2016}).
We also tried the Lorenz maps in \cite{far}, where the map has a cusp, and the result is good (the proof was given in \cite{bahsoun23}).
There are examples where the dynamic hides some features in singularities, such as the one in \cite{wormell_pieceMap}, and we need to add some correction terms to the fast response method.
The convergence of the fast response method actually depends on the integrability of shadowing vectors and renormalized second-order tangent solutions, which might exist somewhere beyond uniform hyperbolicity.

In the highest level of non-hyperbolicity, people can prove that there is no linear response.
One of the examples is the tent map, and Baladi proved that the linear response does not exist, that is, the physical measure is not differentiable with respect to the dynamics \cite{Baladi2007}; there are other examples such as \cite{Wormell2019}.
Failure can also be caused by other reasons, such as intermittency or homoclinic tangencies in systems such as the Henon map \cite{Pujals2000,Fornss1992}.
For these cases, no linear response algorithm works, but they offer good venues to investigate how different mechanisms fail the fast response algorithm differently.

Most practical or physical systems fall into the large gray area in terms of the level of non-hyperbolicity, where we do not yet have a proof or disproof of the linear response, and no existing algorithm seems to work accurately.
Some examples are modeled from physical phenomena, such as the Lorenz 96 system.
People have different opinions about whether the linear response exists.
The author's opinion is that either linear response does not exist and we must have approximations, or some approximations would significantly improve the cost/error.

The first approximation is to compute only the shadowing but not the unstable part of the linear response, since shadowing is less affected by non-hyperbolicity. 
We tried this on the Lorenz 63 system \cite{Ni_NILSS_JCP,Ni_nilsas} and a 3d fluid problem \cite{Ni_CLV_cylinder}, and the shadowing contribution alone gives good results.
On the other hand, the unstable part of the linear response is more affected by the non-hyperbolicity, and that part of code does not converge well.
For more chaotic systems, such as the Lorenz 96 system, the result is not good: the shadowing code converges, but the systematic error due to missing the unstable part is large.

The other approach for approximation is to add noise and incorporate another linear response formula, the kernel-differentiation formula (also known as the Cameron-Martin-Girsanov method or the likelihood ratio method).
But we can not just add a simple uniform noise, which would incur either a large error or a high computational cost; the reasons are explained in \cite{Ni_kd}.
We should add a large but local noise, only at locations where the hyperbolicity is bad.
Then we use the kernel-differentiation formula where the noise is large, so the sampling error is small.
Where the noise is small, we use the fast-response algorithm, which is efficient regardless of noise scale, but requires hyperbolicity.

We proposed this program in \cite{Ni_kd}. 
We call it the `triad' program since it involves the path-perturbation method (this is highly related to the shadowing contribution in this paper), the divergence method (this is highly related to the unstable contribution in this paper), and the kernel-differentiation method.
This program requires transferring information between the kernel-differentiation formula and the fast response formulas.

To achieve this, we recently gave the path-kernel formula for combining the kernel-differentiation with the path-perturbation method for linear responses \cite{dud,apk}. 
We also gave the divergence-kernel formula for combining the kernel-differentiation with the divergence method \cite{divKer,DKlinR}.
The path-kernel and divergence-kernel are not only for extending the fast response method; they are useful tools by themselves.
For example, the path-kernel method solves a difficult version of variational data assimilation problem \cite{apk}, and the divergence-kernel method offers a much more straightforward framework for diffusion models \cite{DKlinR}.
However, the path-kernel and divergence-kernel methods still have shortcomings: path-kernel is expensive for small noise systems with very strong expanding directions; divergence-kernel is expensive for small noise systems with very strong contracting directions.

Hence, to overcome these shortcomings of all the 2-in-1 methods mentioned above, we do need to combine the kernel-differentiation with our path-divergence method, or the fast response method.
This triad program, once done, should compute the linear response of random dynamical systems with minimal cost, or compute the approximation of linear response of deterministic system with best accuracy under minimal cost.

\section{A 21-dimensional numerical example} \label{s:example}

This section illustrates the fast response algorithm on a 21-dimensional solenoid map with 20 unstable dimensions.
Here $\cM = \R \times \T^{20}$, and the governing equation is
\[ \begin{split}
  x^1_{n+1} &= 0.05x^1_{n} + \gamma + 0.1 \sum_{i = 2}^{21} \cos(5x^i_{n}) \\
  x^i_{n+1} &= 2x^i_{n} + \gamma (1+x^1_{n}) \sin(2x^i_{n}) \mod 2\pi, \quad \textnormal{for} \quad 2\le i\le 21
\end{split} \]
where the superscript labels the coordinates.
The perturbation is caused by changing $\gamma$, and the instantaneous objective function is 
\[ \begin{split}
  \Phi(x) := (x^1)^3 + 0.005 \sum_{i = 2}^{21} (x^i - \pi)^2.
\end{split} \]

The default setting, $N=20$ steps in each segment, $A=200$ segments, $\gamma=0.1$, and $W=10$, is used unless otherwise noted.
The code is at \url{https://github.com/niangxiu/fr}.
Figure~\ref{f:orbit} shows a typical orbit.
Figure~\ref{f:change A} shows that the variance of the computed derivative is proportional to $A^{-0.5}$.
Figure~\ref{f:change W} shows that the bias in the averaged derivative decreases as $W$ increases,
but the variance increases as $W^{0.5}$,
indicating that we should increase $A$ together with $W$.
This square-root trend verifies the assumption we made for the error estimation in section~\ref{s:cost estimation}.
Finally, figure~\ref{f:change parameter} shows that 
the derivative computed by fast response correctly reflects the trend of the objective as $\gamma$ changes.

\begin{figure}[ht] \centering
  \includegraphics[width=0.40\textwidth]{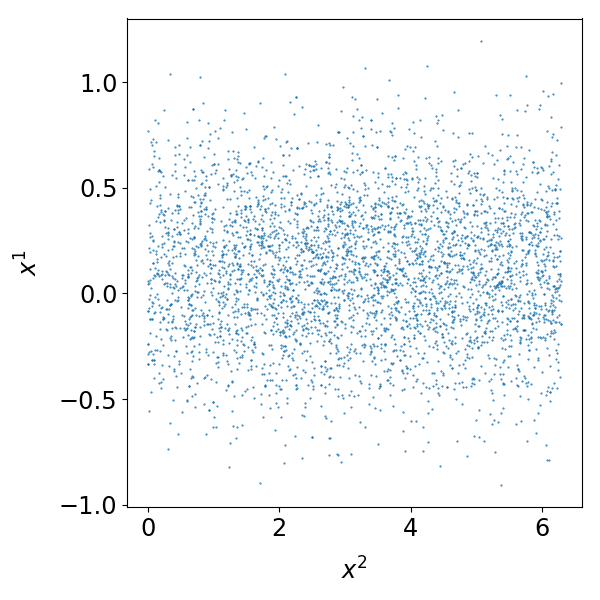}
  \includegraphics[width=0.40\textwidth]{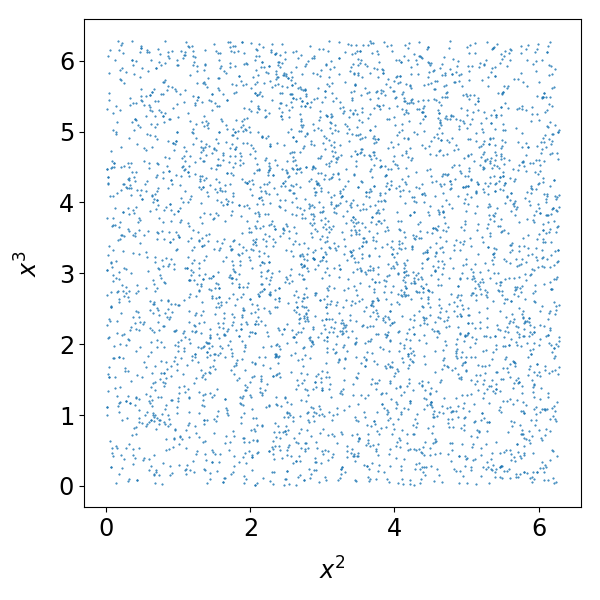}
  \caption{The empirical measure of a orbit with default setting.}
  \label{f:orbit}
\end{figure}

\begin{figure}[ht] \centering
  \includegraphics[width=0.45\textwidth]{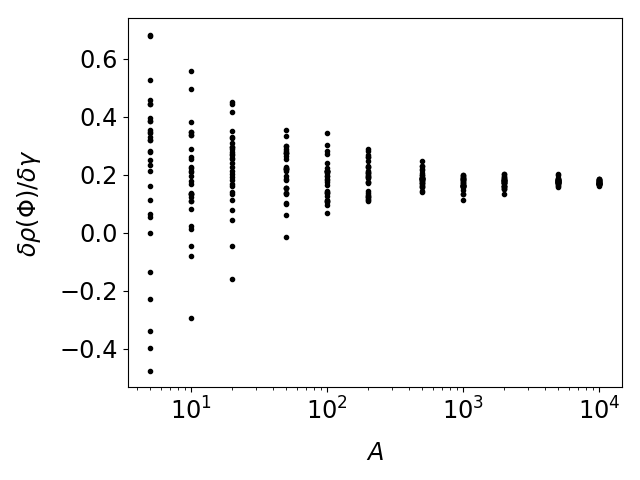}
  \includegraphics[width=0.45\textwidth]{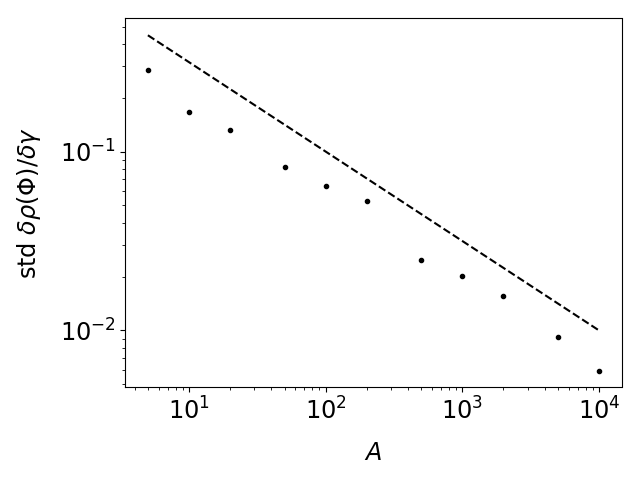}
  \caption{Effects of $A$. Left: derivatives from 30 independent computations for each $A$.
  Right: the sample standard deviation of the computed derivatives,
  where the dashed line is $A^{-0.5}$.}
  \label{f:change A}
\end{figure}

\begin{figure}[ht] \centering
  \includegraphics[width=0.45\textwidth]{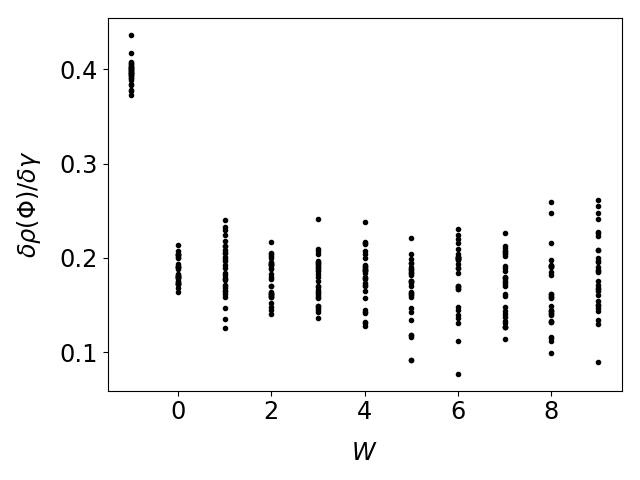}
  \includegraphics[width=0.45\textwidth]{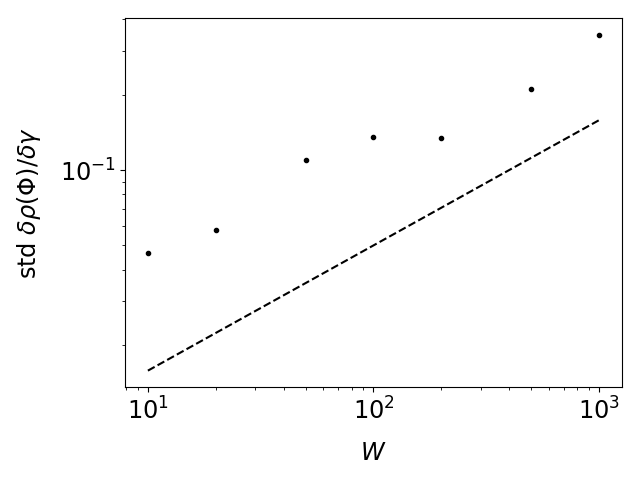}
  \caption{Effects of $W$. Left: derivatives computed by different $W$'s.
  Right: standard deviation of derivatives, where the dashed line is $0.005W^{0.5}$.}
  \label{f:change W}
\end{figure}

\begin{figure}[ht] \centering
  \includegraphics[width=0.5\textwidth]{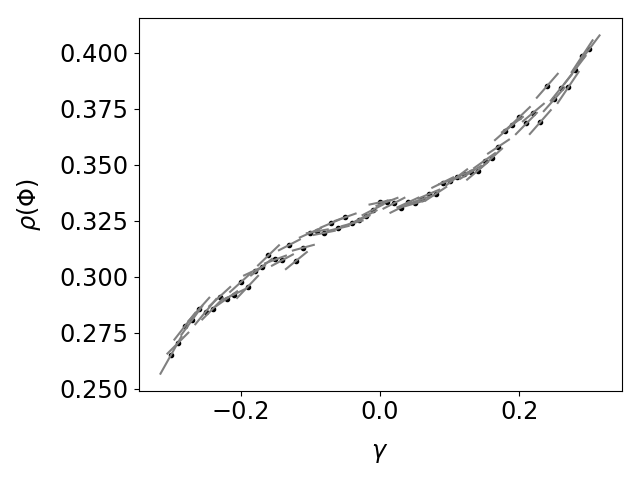}
  \caption{Averaged objectives and derivatives for different parameter $\gamma$.
  The grey lines are the derivatives computed by fast response.}
  \label{f:change parameter}
\end{figure}

On a single-core 3.0GHz CPU, for $10^4$ segments, which is a total of $2\times10^5$ steps,
the time for computing the orbit is 7.4 seconds using non-compiled python code;
whereas the fast response algorithm on the same orbit takes another 47 seconds.
So, the time cost is about 6 times of simulating the orbit.
In fact, our implementation of the fast response algorithm is not optimal; 
for example, we still use dense matrix representations for $f_*$ and $\nabla f_*$.
Our algorithm can run even faster if we pre-compile the code or use the sparsity of the system, either via sparse matrices or graph tracing.

Our example is difficult for the path-perturbation method.
With $W=10$, the magnitude of the integrand is
\[ \begin{split}
  \|f^W_* X (\Phi)\|\sim \lambda_{max}^W \|X\| \|d\Phi\|
  \sim 2^{10} \times \sqrt{21} \times 1 \sim 4700,
\end{split} \]
In order to bring the standard deviation below $7\times 10^{-3}$,
the number of sample orbits required is about $4.5\times 10^{11}$.
Hence, the total number of steps computed is about $4.5\times 10^{12}$,
where each step contains one application of $f$ and one step of the first order tangent equation.
In comparison, as shown in figure~\ref{f:change A},
the same setting requires running the fast response for $2\times 10^5$ steps,
where each step contains one application of $f$, about 30 first and second order tangent equations.
Say that the cost per step is 6 times the path-perturbation algorithms,
then, overall, the fast response is about $10^{6}$ times faster.

Our example is also difficult for divergence or transfer operator algorithms, where SRB measures are approximated on a basis.
Since the attractor has 20 dimensions, even with adaptive meshing,
and each dimension has only 10 elements, the entire attractor needs about $10^{20}$ basis elements.
Assuming that each element only needs one float number to store, it would cost $10^{10}$ GB only to store the elements.

Our example is also difficult for previous blended algorithms.
For the default $\gamma$, the shadowing contribution is about 0.4, and the unstable contribution is about 0.2.
Algorithms that use approximations of the unstable contribution,
such as nonintrusive shadowing and blended response, have a large error.
Also note that if we did not use a small coefficient in the unstable part of $\Phi$,
then the unstable contribution takes the largest portion:
this is predicted in \cite{Ruesha} for cases where $u/M$ is large.

Finally, we compare with regression or surrogate-based methods.
In figure~\ref{f:change parameter}, because $\rho(\Phi)$ computed from finite orbits has oscillations, 
revealing the correct trend between $\rho(\Phi)$ and $\gamma$ takes about 20 data points with different $\gamma$.
Although it is hard to quantify the error in such regression methods,
which would require a probability distribution on regression models;
we can still roughly say that, overall, the fast response is a few times faster than regression methods.

The cost comparison with regression methods is very encouraging.
In the simplest non-chaotic situations, where we want the derivative of a function,
computing the derivative function is typically a few times faster than regression.
The fast response algorithm recovers such cost comparison for chaotic situations:
this hints that the fast response could perhaps be close to the best possible efficiency.
With further work, our algorithm can be transformed into an adjoint algorithm whose cost is almost independent of the number of parameters.
Many practical problems have many parameters, so the fast adjoint response algorithm can be much faster than regression methods \cite{far}.

\section*{Acknowledgements}

I am in great debt to Yi Lai for discussions on the geometric aspect of this problem.
I also thank David Ruelle, Dmitry Dolgopyat, Alansari Nawaf, Charles Pugh, Peidong Liu, and Stefano Galatolo for helpful discussions.
This research is supported by the China Postdoctoral Science Foundation 2021TQ0016,
the International Postdoctoral Exchange Fellowship Program YJ20210018,
and the Richman Fellowship from the math department of UC Berkeley.

\section*{Data Availability Statements}

The data used in the numerical example are generated by the code which is publicly available at \url{https://github.com/niangxiu/fr}.

\appendix
\section{Pushforward operators as tensors} \label{a:f*}

In $\R^M$, the pushforward operator $f_*$ is a matrix.
In particular, when differentiating a composition of several pushforwards, the Leibniz rule applies.
In this section,
we establish the Leibniz rule for general pushforward operators on Riemannian manifolds.
To achieve this, we will define the pushforward operator on vectors as a $(1,1)$-tensor and define its Riemannian connection.
Finally, we extend this Leibniz rule to $u$-vectors.

\begin{definition} [$f_*$]
\label{d:f*}
  Let $f:\cM_1\rightarrow \cM_2$ be a $C^\infty$ diffeomorphism.
  Let $e\in \mathfrak{X}(\cM_1)$ be a $C^\infty$ vector field over $\cM_1$; 
  $ \alpha\in\mathfrak{X}(\cM_2)^* $ be a $C^\infty$ 1-form over $\cM_2$.
  Define
  \[ \begin{split}
    f_*(\alpha, e) := \alpha (f_* e) = e (f^* \alpha).
  \end{split} \]
\end{definition}

Let $z_1$, $z_2$ be coordinates on $\cM_1$, $\cM_2$, respectively.
Written in coordinates, we have
\[ \begin{split}
  f_* = \pp{}{z_2^j} f^j_i dz_1^i ,\quad
  f_* (\alpha, e) = (\alpha \pp{}{z_2^j}) f^j_i (e dz_1^i ) ,
\end{split} \]
where $f^j_i$ is the Jacobian matrix under $z_1$ and $z_2$.
According to our definition, $f_*$ is a tensor field, in the sense that it is a $C^\infty$-multilinear function $f_*:\mathfrak{X} (\cM_1) \times \mathfrak{X}^*(\cM_2) \rightarrow C^\infty(\cM_2)$.
We then define the Riemannian connection of this pushforward operator.

For $\pp {} q \in T\cM_1$, the textbook way to define connections of tensors is
\[ \begin{split}\label{e:dandan}
  (\nabla_{\pp{}q} f_*) ( \alpha, e)
  :=& ({f_* \pp{}q }) (\alpha f_* e) 
  - (\nabla_{f_* \pp{}q } \alpha) f_* e
  - \alpha f_* \nabla_{\pp{}q} e 
  = \alpha \nabla_{f_* \pp{}q } (f_* e) - \alpha f_* \nabla_{\pp{}q} e,
\end{split} \]
where the second equality uses the definition of derivative of covectors, $\nabla (\omega X) = \omega \nabla X + (\nabla \omega) X$.
Hence, $(\nabla_{(\cdot)} f_*) (\cdot,\cdot)$ is a tensor field in the sense that it is a $C^\infty$-multilinear function
$:\mathfrak{X} (\cM_1) \times \mathfrak{X} (\cM_1) \times \mathfrak{X}^*(\cM_2) \rightarrow C^\infty(\cM_2)$.
Notice that, similar to typical Riemannian connections,
$\nabla f_*$ only requires the value of $\pp{}q$ at a point.
Since $\alpha$ is a common factor, we may neglect it on both sides of the equation.
Hence, we define the Riemannian connection of the pushforward operator $(\nabla_{(\cdot)} f_*) (\cdot):\mathfrak{X} (\cM_1) \times 
\mathfrak{X} (\cM_1) \rightarrow \mathfrak{X}(\cM_2)$, as
\[ \begin{split} 
  (\nabla_{ \pp{}q } f_*) e
  := \nabla_{f_* \pp{}q } (f_* e) - f_* \nabla_{\pp{}q} e.
\end{split} \]
To write $\nabla f_*$ in coordinates, by the coordinate form of $f_*$, we have
\[ \begin{split}
  (\nabla_{ \pp{}q } f_*) e
  = (\nabla_{f_* \pp{}q } \pp{}{z_2^j}) f^j_i (dz_1^i e)
  + (\pp{}{z_2^j}) (d f^{j}_i \pp{} q) (dz_1^i e)
  + (\pp{}{z_2^j}) f^j_i (e \nabla_{\pp{}q} dz_1^i).
\end{split} \]

\begin{lemma}[Leibniz rule for composition of pushforward operators]
\label{l:chala}
  Let $g:\cM_2\rightarrow \cM_3$ be a diffeomorphism.
  Then
  \[ \begin{split}
  \nabla_{g_* f_* \pp{}q } (g_* f_* e) 
  = 
  (\nabla_{f_* \pp{}q} g_*) f_* e 
  + g_* (\nabla_{\pp{}q} f_*) e 
  + g_* f_* \nabla_{\pp{}q} e .
  \end{split} \]
\end{lemma}

\begin{remark*}
Besides the proof below,
readers may find it consolidating to prove by writing everything in coordinates,
which also helps us to check that the coordinate form is correct.
When doing that, use the following relations to cancel or combine terms:
  \[ \begin{split}
    \pp{}{z_2^l} \nabla_{\pp{}q} dz_2^j 
    + dz_2^j \nabla_{\pp{}q} \pp{}{z_2^l}
    = \pp{}q \delta^j_l 
    = 0, \quad
    e \nabla_{\pp{}q} dz^i_1
    + dz^i_1 \nabla_{\pp{}q} e
    =  
    \pp{}q (e dz^i_1).
  \end{split} \]
\end{remark*}

\begin{proof}
  By definition,
  \[ \begin{split}
    \nabla_{g_* f_* \pp{}q } (g_* f_* e)
    = \nabla_{g_* (f_* \pp{}q) } (g_* (f_* e))
    = (\nabla_{f_*\pp{}q } g_*) (f_* e) + g_* \nabla_{f_* \pp{}q } (f_* e)\\
    = (\nabla_{f_*\pp{}q } g_*) (f_* e)
    + g_* (\nabla_{\pp{}q} f_*) e 
    + g_* f_* \nabla_{\pp{}q} e.    
  \end{split} \]
\end{proof}

Finally, we extend the Leibniz rule to the case where $e=e_1\wedge\cdots\wedge e_u$ is a $u$-vector.
Now $f_*e = f_* e_1 \wedge \cdots \wedge f_* e_u$, and 
\[ \begin{split}
  &\nabla e = \sum_{i=1}^u e_1 \wedge \cdots \nabla e_i \wedge \cdots \wedge e_u ,\quad
  \nabla (f_*e) = \sum_{i=1}^u f_* e_1 \wedge \cdots \nabla(f_*e_i) \wedge \cdots \wedge f_* e_u.
\end{split} \]

\begin{definition} [$\nabla f_*$ on $u$-vectors] \label{d:dnabla}
The Riemannian connection of pushforward operators on $u$-vectors,
\[ \begin{split}
  (\nabla_{ \pp{}q }  f_*)e := \nabla_{f_* \pp{}q } (f_*e) - f_*\nabla_{ \pp{}q } e
  = \sum_{i=1}^u f_* e_1 \wedge \cdots (\nabla_{ \pp{}q } f_*) e_i 
  \wedge \cdots \wedge f_* e_u .\\
\end{split} \]
In the last expression, $\nabla_{\pp{}q} f_*$ acts on single vectors, as defined earlier.
Notice that by our definition, 
$(\nabla_{ \pp{}q }  f_*)e $ does not depend on the choice of $e_i$'s,
as long as their wedge product is $e$.
\end{definition}

\begin{lemma} [Leibniz rule for differentiating $u$-vectors] \label{l:Leibniz}
  Let $e$ be a $C^\infty$ $u$-vector, then
  \[ \begin{split}
    \nabla_{g_* f_* \pp{}q } (g_* f_* e) 
    &= (\nabla_{f_* \pp{}q} g_*) f_* e 
    + g_* (\nabla_{\pp{}q} f_*) e 
    + g_* f_* \nabla_{\pp{}q} e ,\\
    \nabla_{f_*^k \pp{}q } (f_*^k e)
    &= \sum_{n=0}^{k-1} f_*^{k-n-1} (\nabla_{f_*^n \pp{}q} f_*) f_*^n e 
    + f_*^k \nabla_{\pp{}q} e .\\
  \end{split} \]
\end{lemma}

\begin{remark*}
  (1) There is no need to add a subscript to $\nabla f_*$ to indicate steps,
  since the two vectors it applies to, $f_*^n e$ and $f_*^n \pp{}q$,
  already well-locate this tensor.
  (2) We use $C^\infty$ in this section only because `$C^\infty$-multilinear' is a conventional
  terminology in differential geometry textbook.
  In fact, we only require the differentiation to be meaningful at the particular point
 and in the particular direction. 
  For example, the second equation of lemma
  applies to the rough $u$-vector field $e$ 
  defined in equation~\eqref{e:e}, which is differentiable only in unstable directions.
\end{remark*} 

\begin{proof}
  Inductively apply $\nabla (f_*e) = (\nabla f_*)e + f_*\nabla e$.
\end{proof}

\section{Derivative-like \texorpdfstring{$u$}{u}-vectors}
\label{a:derivative}

Consider a $u$-vector field $e:=e_1\wedge\cdots\wedge e_u$
smoothly defined along a direction $\pp{}q$; the derivative is 
\begin{equation*}
\nabla_{\pp{}q} e  = \sum_i e_1\wedge\cdots \nabla_{\pp{}q} e_i \cdots\wedge e_u.
\end{equation*}
This expression takes the specific form that each term has a special item $\nabla_{\partial/\partial q} e_i$, which usually is not within the span of $\{e_i\}_{i=1}^u$.
However, it is not trivial that a summation of these derivatives still has this form,
especially when the directions $\pp{}q$ and the basis of $e$ in each summand are different.
Motivated by this, we define a subspace of $u$-vectors which looks like these derivatives,
and show some related properties.

\begin{definition} \label{d:derivative like u vectors}
  At point $x$,
  the collection of derivative-like $u$-vectors of $e = \wedge^u_{i=1}e_i$, 
  written on the basis $\{e_i\}_{i=1}^u$, is defined as
  \[ \begin{split}
    \cD^u := \{p\in\wedge^u T_x \cM:
    p= \sum_i e_1\wedge\cdots\wedge p_i \wedge\cdots\wedge e_u,\;
    p_i\in T_x\cM\}.
  \end{split} \]
\end{definition}

In our paper, the basis $\{e_i\}_{i=1}^u$ is typically the basis of the unstable subspace.
Letting $\{e_i\}_{i=1}^M$ be a full basis of $T_x\cM$ whose first $u$ vectors spans $V^u$,
then by decomposing $p$ onto the basis of $\wedge^u T_x \cM$,
we can see that $\cD^u$ is the direct sum
\begin{equation} \begin{split} \label{e:passionate}
  \cD^u = \spanof \{e\} \oplus \sum_{i=1}^u \sum_{j>u}
  \spanof\{e_1\wedge\cdots\wedge e_{i-1} \wedge e_{i+1}
  \wedge\cdots\wedge e_u\wedge e_j\}.
\end{split} \end{equation}
Hence, under a given basis, 
$p_i$ is unique modulo $\spanof \{e_i\}_{i=1}^u$.

So long as $\spanof \{e_1,\cdots,e_u\}$ is the same, 
$\cD^u$ as a subspace is independent of the selection of basis.
To see this, let $\{e'_i\}_{i=1}^M$ be another full basis such that
$\spanof \{e_1,\cdots,e_u\} = \spanof \{e'_1,\cdots,e'_u\}=V^u$,
and let $\cD_{e'}$ be the corresponding subspace.
Write each term in \cref{e:passionate} by the other basis, we can see that
\[ \begin{split}
  &\spanof\{e_1\wedge\cdots\wedge e_{i-1} \wedge e_{i+1}
  \wedge\cdots\wedge e_u\wedge e_j\}\\
  \subset
  &\spanof \{e'\} \oplus \sum_{i=1}^u \sum_{j>u}
  \spanof\{ e'_1\wedge\cdots\wedge  e'_{i-1} \wedge  e'_{i+1}
  \wedge\cdots\wedge e'_u\wedge  e'_j\}.
\end{split} \]
Hence, $\cD^u \subset \cD_{ e'}$.
By symmetry, $\cD^u = \cD_{ e'}$.

The next step is to consider how to express derivative-like $u$-vectors 
under a new basis of the same span.
Should $p$ indeed be the derivative of a $u$-vector field,
there is a formula for changing to a new basis
which is a constant linear combination of the old basis.
We will show that the same formula is true when $p$ is only derivative-like but not really a derivative.

\begin{lemma} [change of basis formula] \label{l:change of basis}
  Let $\{ e'_j\}_{j=1}^u$ be another basis of $V^u$ such that $e'_j = a_j^i  e_i$;
  let $p_k, p'_k\in T_x\cM$ satisfy $p'_k := a_k^i p_i$ (notice that typically $p_k, p'_k \notin V^u$), then
  \[ \begin{split}
    \sum_{k=1}^u  e'_1\wedge\cdots\wedge  p'_k \wedge\cdots\wedge  e'_u 
    = \det (a_i^j) 
    \sum_{k=1}^u e_1\wedge\cdots\wedge p_k \wedge\cdots\wedge e_u .
  \end{split} \]
In other words, let $[p]_e:=[p_1, \ldots, p_u]$ be the matrix of $p=\sum_{k=1}^u e_1\wedge\cdots\wedge p_k \wedge\cdots\wedge e_u$ under the basis $\underline e:=[e_1, \ldots, e_u]$.
Then, under a new basis $\underline e ' = \underline e A$, $A\in \R^{u\times u}$, 
\[ \begin{split}
  [p]_{e'} 
  = \det(A)^{-1} [p_1', \ldots, p_u']
  = \det(A)^{-1} [p]_e A.
\end{split} \]
\end{lemma}

\begin{remark*}
(1) For the simplified case, assume that $e$ and $e'$ are differentiable, $A$ is constant, $\underline e' = \underline e A$, for some vector $Y$, let 
$p_k:=\nabla_Y e_k$, $p'_k:=\nabla_Y e'_k$, then $p = \nabla_Y e$, $e' = \det(A) e$, and
\begin{equation*}
    \underline p' = \nabla_Y \underline e'
    = (\nabla_Y \underline e) A
    = \underline p A,
    \quad
    [p]_e = \underline p,
\end{equation*}
which satisfies the assumption in the lemma.
Then this lemma is just the fact that 
\begin{equation*}
p = \nabla_Y e = \det(A)^{-1} \nabla_Y e',
\quad 
[p]_{e'} = \det(A)^{-1} \underline p' = \det(A)^{-1} \underline p A = \det(A)^{-1} [p]_e A.
\end{equation*}
(2)
We use this lemma to prove \cref{l:titan},
where $p$ is a summation of derivatives.
For this particular case, we may as well first prove the lemma for each summand,
then apply linearity to obtain the lemma for the summation.
However, here we give a more general and algebraic proof,
which does not rely on the fact that p is a derivative or a summation of derivatives.
\end{remark*}

\begin{proof}
  We can reorder $ e'_i $ so that for all $i$,
  $S_i:=\{ e'_1,\cdots, e'_i, e_{i+1},\cdots,e_u\}$ is an independent set of vectors.
  Changing basis from $S_0$ to $S_u$ can be achieved by sequentially changing from 
  $S_i$ to $S_{i+1}$, where only one vector is changed in each step.
  Hence, by induction, it suffices to show that the lemma is true for any one step,
  for example, the first step.
  Hence, it suffices to prove for the case where $ e'_1 = a^i e_i$, 
  and $ e'_j = e_j$ for all $j\ge 2$.
  Now $p'_1=a^ip_i$, $p'_j = p_j$ for all $j\ge 2$.
  the left hand side is
  \[ \begin{split}
    &LHS = \sum_{k=1}^u  e'_1\wedge\cdots\wedge  p'_k \wedge\cdots\wedge  e'_u 
    =  p'_1\wedge e'_2 \wedge\cdots\wedge  e'_u 
    + \sum_{k\ge 2}^u  e'_1\wedge\cdots\wedge  p'_k \wedge\cdots\wedge  e'_u 
    \\
    =& \sum_i a^i p_i \wedge e_2 \wedge \cdots \wedge e_u
    + \sum_{k\ge 2} \sum_i a^i e_i\wedge e_2\wedge\cdots\wedge p_k \wedge\cdots\wedge e_u 
    \\
    =& \sum_{i} a^i p_i \wedge e_2 \wedge \cdots \wedge e_u
    +  a^1 \sum_{k\ge 2} e_1 \wedge \cdots p_k \cdots \wedge e_u
    + \sum_{k\ge 2} \sum_{i\ge2} a^i e_i\wedge e_2\cdots\wedge p_k \wedge\cdots\wedge e_u
  \end{split} \]
  In the last summation, notice that an exterior product vanishes if $e_i$ appears twice.
  Hence,
  \[ \begin{split}
    &LHS = 
    a^1 \sum_{k\ge 2} e_1 \cdots p_k \cdots e_u
    + \left( \sum_{i} a^i p_i \wedge e_2 \cdots e_u
    + \sum_{k\ge 2} a^k e_k \wedge e_2\cdots p_k \cdots e_u \right)
  \end{split} \]
  Comparing the two summations in parentheses, 
  notice that interchanging the positions of $p_k$ and $e_k$ in an exterior product changes the sign,
 and hence all terms cancel except the term with $i=1$, and 
  \[ \begin{split}
    &LHS = 
    a^1 \sum_{k\ge 2} e_1 \wcw p_k \wedge \cdots \wedge e_u
    + \sum_{i=1} a^i p_i \wedge e_2 \wedge \cdots \wedge e_u\\
    =& a^1 \sum_k e_1 \wedge \cdots \wedge p_k \wcw e_u = a^1 p.
  \end{split} \]
  Since $a^1$ is the determinant of our current transformation matrix,
  we have proved the lemma for one step.
  The lemma is proved by induction.
\end{proof}

\section{Projection operators on derivative-like \texorpdfstring{$u$}{u}-vectors} 
\label{a:projection}

For derivative-like $u$-vectors, $\cD^u$, defined in appendix~\ref{a:derivative},
only one entry is not in $V^u$ in each wedge product;
hence, we can extend the definition of projection operators from one-vectors to $\cD^u$,
by applying the projection on the one exceptional entry in each wedge product.
Notice that, in general, we can not extend our definitions to $\wedge^u T_K\cM$ while keeping all the good properties; we must work in $\cD^u$.
In this section, $e\in\wedge^u V^u$ by default.

\begin{definition} \label{d:projection}
The projection operator $P$ on $p\in\cD^u$ is
\[ \begin{split}
  P p = \sum_i e_1\wedge\cdots\wedge P p_i \wedge\cdots\wedge e_u ,
\end{split} \]
where $P$ on the right side is the projection operator on one-vectors.
\end{definition}

The projection operators used in this paper are $P^u, P^s$, $P^\parallel$, and $ P^\perp$.
The first two operators are oblique projections along stable or unstable subspace 
to the unstable or stable subspace, respectively.
The last two operators are orthogonal projections 
onto the unstable subspace and its orthogonal complement.
These operators, applied on one-vectors, are illustrated in figure~\ref{f:projection}.
Notice that computing $P^u$ and $P^s$ both require both $V^u$ and $V^s$;
in contrast, computing $P^\parallel$ and $ P^\perp$ only require $V^u$, thus are faster.
We have the decomposition
\begin{equation*}
P^u+P^s=Id,
\quad
P^\perp+P^\parallel=Id.
\end{equation*}
For both single and $u$-vectors, denote
\[ \begin{split}
  p^u := P^up, \quad
  p^s := P^sp; \quad
  p^\parallel := P^\parallel p, \quad
  p^\perp := P^\perp p.
\end{split} \]

\begin{figure}[ht] 
  \centering
  \includegraphics[width=2.5in]{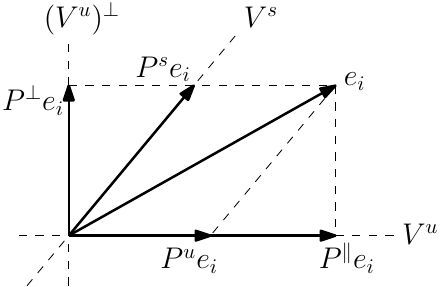}
  \caption{$P^u, P^s$, and $P^\parallel, P^\perp$ applied on $e_i$.}
  \label{f:projection}
\end{figure}

For fixed $V^u$, $P^\parallel$ and $P^\perp$ do not depend on the choice of basis.
To see this, note that 
\[ \begin{split}
  \cD^{u\perp} := \sum_{i=1}^u \sum_{j>u}
  \spanof\{e_1\wedge\cdots\wedge e_{i-1} \wedge e_{i+1}
  \wedge\cdots\wedge e_u\wedge e_j\}.
\end{split} \]
is the same as long as $e_{u+1},\cdots,e_M$ is a basis of $(V^u)^\perp$;
then note that $P^\parallel$ and $P^\perp$ are projections corresponding to the decomposition
$\cD^u = \spanof \{e\}\oplus \cD^{u\perp}$.
Similarly, if both $V^u$ and $V^s$ are fixed, 
$P^u$ and $P^s$ do not depend on the choice of basis.

\begin{lemma}[composing projections with addition]
For $p,  p' \in \cD^u$,
\[ \begin{split}
  P(p +  p') = Pp + P p'.
\end{split} \]
\end{lemma}

\begin{lemma}[composing projection operators]\label{l:compose projection}
\[ 
  P^\parallel P^u = P^u,
  P^u P^\parallel = P^\parallel ;\;
  P^\perp P^s = P^\perp,
  P^s P^\perp = P^s ;\;
  P^\perp P^u = P^s P^\parallel= 0.
\]
\end{lemma}
\begin{remark*}
  Notice that, typically $P^u P^\perp \neq 0$, $P^\parallel P^s \neq 0$.
\end{remark*}

\begin{lemma}[composing projections with pushforwards] \label{l:P with f}
\[ \begin{split}
  f_*P^u = P^uf_* = P^uf_*P^u ,\quad 
  f_*P^s = P^sf_* = P^sf_*P^s ;\\
  f_* P^\parallel = P^\parallel f_* P^\parallel ,\quad
  P^\perp f_* = P^\perp f_* P^\perp .
\end{split} \]
Here the projection $P$'s are evaluated at suitable locations.
\end{lemma}

\begin{proof}
The first three equalities are due to the invariance of the stable and unstable subspace.
The last equality is because
\[ \begin{split}
P^\perp f_* = P^\perp P^s f_* 
= P^\perp P^s f_* P^s
= P^\perp P^s f_* P^s P^\perp 
= P^\perp P^s f_* P^\perp
= P^\perp f_* P^\perp . 
\end{split} \]
\end{proof}

\begin{lemma}[expressing $P^\parallel$ by inner products] \label{l:baozi}
For any $p\in \cD^u$,
\[ \begin{split}
  P^\parallel p
  = \ip{p, e} \frac{e}{\|e\|^2};\quad
  P^\perp p
  = p - \ip{p, e} \frac{e}{\|e\|^2}.
\end{split} \]
\end{lemma}

\begin{proof}
For the first equation,
since both sides are in $\wedge^u V^u$,
which is a one-dimensional subspace,
it suffices to prove the equation after taking inner product with $e$, that is,
\[ \begin{split}
  \ip{P^\parallel p, e}
  = \ip{p, e} \frac{\ip{e,e}}{\|e\|^2}
  = \ip{p, e}.
\end{split} \]
Further adding $\ip{P^\perp p, e}=0$ to the left proves this equality.
The second equality is because $P^\perp+P^\parallel=Id$,
hence $P^\perp p = p - P^\parallel p$.
\end{proof}

\bibliography{library}
\bibliographystyle{abbrv}

\end{document}